\documentclass[12pt]{amsart}
\usepackage{amssymb,amsmath,amsthm, geometry, accents, enumerate}
\usepackage[shortlabels]{enumitem}
\usepackage{mathrsfs}
\usepackage{ color, soul, hyperref,mathtools, faktor}
\usepackage{stmaryrd} 
\usepackage[all]{xy}
\usepackage{tikz-cd} 
\hypersetup{
    colorlinks=true,
    linkcolor=blue,
    filecolor=magenta,      
    urlcolor=cyan,
    citecolor=blue,
}

\usepackage{cleveref}
\usepackage[scr=boondox]   
           {mathalpha}
\everymath{\displaystyle}

\usepackage{todonotes}

\newcommand*\xbar[1]{%
  \hbox{%
    \vbox{%
      \hrule height 0.5pt 
      \kern0.5ex
      \hbox{%
        \kern-0.1em
        \ensuremath{#1}%
        \kern-0.1em}}}}
      
\def\temp{&} \catcode`&=\active \let&=\temp

\newtheorem{thm}{Theorem}[section]
\newtheorem{cor}[thm]{Corollary}
\newtheorem{lemma}[thm]{Lemma}
\newtheorem*{lemma*}{Lemma}
\newtheorem*{thm*}{Theorem}
\newtheorem{prop}[thm]{Proposition}
\theoremstyle{definition}
\newtheorem{defn}[thm]{Definition}
\newtheorem{eg}[thm]{Example}
\newtheorem{defn-rmk}[thm]{Definition-Remark}
\newtheorem{notationdefinition}[thm]{Definition-Notation}
\theoremstyle{remark}
\newtheorem{rmk}[thm]{Remark}

\DeclareMathOperator{\width}{W}

\DeclareMathOperator{\Ann}{Ann}
\DeclareMathOperator{\ch}{char}
\DeclareMathOperator{\hd}{hd}
\DeclareMathOperator{\grade}{grade}
\DeclareMathOperator{\Supp}{Supp}
\DeclareMathOperator{\pd}{pd}
\DeclareMathOperator{\proj}{Proj}

\DeclareMathOperator{\Ext}{Ext}
\DeclareMathOperator{\Hom}{Hom}

\DeclareMathOperator{\supph}{Suppi}
\renewcommand{\min}{\mbox{\rm min}}
\renewcommand{\max}{\mbox{\rm max}}
\DeclareMathOperator{\im}{im}
\DeclareMathOperator{\cone}{cone}
\DeclareMathOperator{\spc}{Sp}
\DeclareMathOperator{\bisp}{BiSp}
\DeclareMathOperator{\mods}{mod(\textit{S})}
\DeclareMathOperator{\modr}{mod(\textit{R})}
\DeclareMathOperator{\starmods}{*mod(\textit{S})}

\DeclareMathOperator{\stark}{*\textit{K}}
\DeclareMathOperator{\tildes}{\underaccent{\tilde}{\it{s}}}
\DeclareMathOperator{\tildef}{\underaccent{\tilde}{\it{f}}}

\renewcommand{\mod}{\mbox{mod}}
\newcommand{\CB}{\mathcal{B}}
\newcommand{\F}{\mathcal{F}}
\newcommand{\G}{\mathcal{G}}
\renewcommand{\H}{\mathcal{H}}
\newcommand{\I}{\mathcal{I}}
\newcommand{\OX}{\mathcal{O}_X}
\renewcommand{\O}{\mathcal{O}}
\newcommand{\CT}{\mathcal{T}}
\newcommand{\CU}{\mathcal{U}}
\newcommand{\CW}{\mathcal{W}}
\newcommand{\CZ}{\mathcal{Z}}
\DeclareMathOperator{\sheafext}{\mathcal{E}\kern -.5pt\textit{xt}}
\DeclareMathOperator{\sheafhom}{\mathcal{H}\kern -.5pt\textit{om}}

\newcommand{\T}{\mbox{$\mathscr{T}$}}
\newcommand{\CA}{\mathscr{A}}
\newcommand{\C}{\mathscr{C}}
\newcommand{\V}{\mathscr{V}}
\renewcommand{\L}{\mathscr{L}}
\newcommand{\SE}{\mathscr{E}}

\newcommand{\p}{\mathfrak{p}}
\newcommand{\m}{\mathfrak{m}}

\newcommand{\N}{\mathbb N}
\newcommand{\Z}{\mathbb Z}
\DeclareMathOperator{\Kth}{\mbox{$\mathbb K$}}
\DeclareMathOperator{\GWth}{\mbox{$\mathbb GW$}}

\newcommand{\starp}{\mbox{*$\mathscr{P}$}}
\newcommand{\starf}{\mbox{*$\mathscr{F}$}}
\newcommand{\start}{\mbox{*\textit{T}}}

\newcommand{\DbLT}{\mbox{${D^b_{\tiny \L}}(\mathscr T)$}}

 \newcommand{\Oplus}{\ensuremath{\vcenter{\hbox{\scalebox{1.5}{$\oplus$}}}}}

\begin{document}

\pagenumbering{arabic}

\title{Derived equivalences for chain complexes with support} 
\date{}
\author{Ganapathy Krishnamoorthy and Sarang Sane}

\address{Ganapathy Krishnamoorthy*\\
Department of Mathematics \\
Indian Institute of Technology Madras\\
Sardar Patel Road \\
Chennai India 600036}
\email{ganapathy.math@gmail.com}
\address{Sarang Sane\\
Department of Mathematics \\
Indian Institute of Technology Madras\\
Sardar Patel Road \\
Chennai India 600036}
\email{sarangsanemath@gmail.com}

\begin{abstract}
  For a Serre subcategory $\mathscr L$ and a resolving subcategory $\mathscr A$ of an abelian category, we show that the derived equivalence $D^b(\overline{\mathscr A} \cap \mathscr L) \simeq D^b_{\mathscr L}(\mathscr A)$ holds under certain conditions. We apply this to obtain derived equivalences in the contexts of chain complexes of graded modules or coherent sheaves, with finite $\mathscr A$-dimension, supported on closed sets having eventually finite $\mathscr A$-dimension. Using this, we obtain descriptions of the homotopy fibers in (hermitian) K theory of the restriction maps to certain open sets.
\end{abstract}

\maketitle

\section{Introduction}
Let $R$ be a commutative noetherian ring and $I \subseteq R$ an ideal. Consider the category $\modr$ of finitely generated $R$-modules, $\mod(V(I))$ the full subcategory consisting of $I$-torsion modules and $Ch^b_{V(I)}(\modr)$ the category of bounded chain complexes of $R$-modules with homologies in $\mod(V(I))$. It is a well-known fact that the inclusion functor $Ch^b(\mod(V(I))) \to Ch^b_{V(I)}(\modr)$ induces an equivalence $D^b(\mod(V(I)))\to D^b_{V(I)}(\mod(R))$.  More generally, for an abelian category $\C$ and a Serre subcategory $\L$, with analogous definitions, it is classically known that under suitable hypotheses, the functor $D^b(\L) \to D^b_{\tiny \L}(\C)$ is an equivalence \cite{Grothendieck, Keller}. 

In a previous article \cite{GanapathySane}, we proved that if $\L$ is the full subcategory of $I$-torsion modules where $I$ has eventually finite projective dimension, and $\T$ is a thick subcategory of $\modr$, then $D^b(\T \cap \L) \simeq D^b_{\tiny \L}(\T)$. In \Cref{sec:derived-equivalence-for-abelian-category} of this article, we consider thick subcategories $\T$, Serre subcategories $\L$ of an abelian category $\C$ and the functor $Ch^b(\T \cap \L) \to Ch^b_{\tiny \L}(\T)$. In \Cref{thm:derived-equivalence-main-abelian-category}, we show that if for every non-exact chain complex $X_{\bullet}$ in $Ch^b_{\tiny \L}(\C)$ with $\min(X_\bullet) = m$ and a morphism $f \colon X_m \to Q$ with $Q \in \L$, there exists a chain complex $T_{\bullet} \in Ch^b_{\tiny \L}(\T)$ and a morphism $\alpha \colon T_{\bullet} \to X_{\bullet}$ satisfying the properties as in \Cref{defn:SR-property-for-an-abelian-category}, then $D^b(\T \cap \L) \to D^b_{\tiny \L}(\T)$ is an equivalence of categories. We define $(T_{\bullet},\alpha)$ to be a strong reducer for $(X_{\bullet},f)$ and say that $\L$ has the SR property with respect to $\T$, if a strong reducer exists for every such pair $(X_\bullet,f)$. This definition is a generalization of the definition of the SR property in \cite[Definition 4.1]{GanapathySane} to any abelian category. Earlier special cases of the derived equivalence are in \cite{SandersSane} when $I$ is a complete intersection, and in \cite{Mandal2015}, which proves a sheaf theoretic version of the equivalence for quasi-projective schemes.

A key component in the proof of the derived equivalence is a result of independent interest about the connection between Koszul complexes and Tate resolutions on a sequence of elements $\tildes = s_1 , s_2 , \ldots , s_d \in R$ and their powers $\tildes^u$. Let $K_\bullet(\tildes)$ and $T_\bullet(\tildes)$ denote the Koszul complex and Tate resolution on $\tildes$ respectively. Recall that there is a natural inclusion $K_\bullet(\tildes) \rightarrow T_\bullet(\tildes)$ and a natural chain complex map $\kappa \colon K_\bullet(\tildes^{u}) \rightarrow K_\bullet(\tildes)$. In \cite[Theorem 3.5]{GanapathySane}, we proved that for sufficiently large $u$, the map $\kappa$ factors through $T_\bullet(\tildes^u)$.

In \Cref{sec:graded Tate etc.} of this article, we generalize the above result connecting Koszul complexes and Tate resolutions on powers of sequences to graded rings (\Cref{lem:graded-tate-to-koszul}) when $\tildes$ is a sequence of homogeneous elements. This result has an interesting consequence related to local cohomology. Recall that local cohomology can be defined in two ways, namely as the cohomology obtained by tensoring with the \v{C}ech complex, or as the derived functors of the torsion functor $\Gamma_{I}(\_)$, which are direct limits of $ \varinjlim \Ext^{i}_R(R/I^n , \_)$. 
The following theorem which appears as \Cref{thm:local-cohomology} yields a stronger connection between the \v{C}ech complex and Tate resolutions than previously known :
\begin{thm}
    Let $S$ be a commutative noetherian graded ring, $\tildes = s_1,\dots,s_d$  homogeneous elements in $S$, and $M$ a graded $S$-module. 
    Then, there exists an infinite set $E \subseteq \N$ such that $$\check{C}^{\bullet}(\tildes;M) \cong \varinjlim_E \Hom_S(\start_{\bullet}(\tildes^n),M)$$ for some choice of graded chain complex maps $\start_{\bullet}(\tildes^{n_2}) \to \start_{\bullet}(\tildes^{n_1})$, which lift the natural surjections $S/(\tildes^{n_2}) \to S/(\tildes^{n_1})$ for all $n_1,n_2 \in E$ with $n_2>n_1$.
\end{thm}
Using this, we obtain a different and much more direct proof of the equivalence of the two definitions of local cohomology, than the standard one.

In Sections~\ref{sec:resolving-subcategories} and \ref{sec:resolving-der-eq}, we consider a resolving subcategory $\CA$ of $\C$ (in the sense of \cite{Ruben-Sondre-Roosmalen}, refer to \Cref{defn:resolving-subcategory}). Since projective objects are no longer assumed to be in $\CA$, and $\C$ isn't assumed to have enough projectives, several basic homological algebra statements need to be proved. Thus, we prove results such as the existence of a short exact sequence of resolutions for a short exact sequence of objects in $\C$, the existence of $\dim_{\tiny \CA}$, and the existence of resolving double complexes of a chain complex. Having established these basics, we observe that $\overline{\CA}$ is a thick subcategory and establish the usual derived equivalence $D^b_{\tiny \L}(\CA) \simeq D^b_{\tiny \L}(\overline{\CA})$ for any Serre subcategory $\L$. Further, let $\overline{\CA}^{\leq g}$ be the full subcategory of $\C$ consisting of objects $\{C \in \mathscr C \mid \dim_{\tiny \CA}(C)\leq g\}$. We show that $D^b(\overline{\CA}^{\leq g} \cap \L) \simeq D^b(\overline{\CA} \cap \L)$ whenever $\CA$ has the property that for every $C \in \overline{\CA} \cap \L$, there exists $C' \in \overline{\CA}^{\leq g} \cap \L$ such that there is a epimorphism $C' \to C$ in $\C$.
Now putting all the equivalences together yields \Cref{thm:derived-equivalence-for-abelian-category-resolving-sub} :
\begin{thm}\label{thm:intro2}
    Let $\L$ be a Serre subcategory and $\CA$ a resolving subcategory of $\C$. 
    Suppose $\L$ has the SR property with respect to $\overline{\CA}$, and for each $C \in \overline{\CA} \cap \L$, there exist $C' \in \overline{\CA}^{\leq g} \cap \L$ and an epimorphism $C' \to C$ in $\C$ for some $g \in \N$. Then  $$D^b(\overline{\CA}^{\leq g} \cap \L) \simeq D^b(\overline{\CA} \cap \L) \simeq D^b_{\tiny \L}({\overline{\CA}}) \simeq D^b_{\tiny \L}({\CA}).$$
\end{thm}
In the special case that $\C$ is $\modr$, $\CA$ is the resolving subcategory of projective modules, $\L$ is the Serre subcategory of $I$-torsion modules for a perfect ideal $I$, and $g = \grade(I)$, we see that $\overline{\CA}^{\leq g} \cap \L$ is the subcategory of perfect modules supported on $V(I)$. Further, if $\ch(R)$ is prime, then $D^b(\overline{\CA}^{\leq g} \cap \L) \simeq D^b(\overline{\CA} \cap \L)$. In \Cref{sec: der equiv graded} of this article, we prove similar results when $\C$ is the category of finitely generated graded modules over a commutative noetherian graded ring $S$. Analogous sheaf theoretic results are proved in \cite{Mandal2025} for quasi-projective schemes when $\L$ is the Serre subcategory of coherent sheaves supported on a complete intersection subscheme.



In \Cref{sec: der equiv quasi proj}, we apply \Cref{thm:intro2} to obtain derived equivalences for the category of coherent sheaves over a quasi-projective scheme over a noetherian affine scheme. The main idea is to extend a chain complex of coherent sheaves over a quasi-projective scheme to its projective closure, ensuring that the support of the extended chain complex is contained in the closure of the support of the given chain complex. We use this to show that the SR-property, and hence the derived equivalence, holds for certain Serre subcategories and thick subcategories of the category of coherent sheaves (e.g., see \Cref{thm:SR-prop-holds-cohx}).
One special case recovers the corresponding results in \cite{Mandal2015} and \cite{Mandal2025}. Another special case is when the resolving subcategory is the subcategory of locally free sheaves, denoted by $\V(X)$, and the scheme has prime characteristic, which yields the following theorem:
\begin{thm}\label{thm:intro-3}
Let $X$ be a quasi-projective scheme over a noetherian affine scheme, $V \subseteq X$ a closed set, and $\L = \L_{V}$. Suppose $X$ has prime characteristic and $\I_V$ is perfect with $\grade_X(\I_V) = g$ for some ideal sheaf $\I_V$ defining a closed subscheme structure on $V$. Then $$D^b(\overline{\V(X)}^{\leq g} \cap \L) \simeq D^b( \overline{\V(X)} \cap \L) \simeq D^b_{\tiny \L}( \overline{ \V(X))} \simeq D^b_{\tiny \L}(\V(X)).$$
\end{thm}

In \Cref{sec:result-on-K-theory-and-G-theory}, we describe the homotopy fibers of the maps in $\Kth$-theory and $\GWth$-theory given by restriction to the open set $U = X \setminus V$ (with the hypothesis and notations as in \Cref{thm:intro-3}). In particular, in theorems \Cref{thm:K-thry-fibration} and \Cref{thm:G-thry-fibration}, we state that the following sequences are homotopy fibrations in the category of bispectra:
\begin{align*}
    \Kth(\overline{\V(X)}^{\leq g} \cap \L) \to \Kth(\V(X)) \to \Kth(\V(U)) & \quad  \text{in} \spc\\
    \GWth^{[-g]}(\overline{\V(X)}^{\leq g} \cap \L) \to \GWth(\V(X)) \to \GWth(\V(U)) & \quad \text{in} \bisp
\end{align*}
This generalizes earlier results in \cite{Mandal2025} where $V$ is a complete intersection subscheme. Since the details are the same as in \cite{Mandal2025}, we omit proofs.

    \section{Preliminaries}
    Throughout the article, we set the following notations: let $S = \Oplus_{i \in \mathbb Z} S_i$ be a noetherian graded ring and $\starmods$ the category of graded $S$-modules whose objects are finitely generated graded $S$-modules and morphisms between two graded $S$-modules are degree zero graded $S$-module homomorphism. Recall that a graded chain complex of $S$-modules is a chain complex of graded modules with all differential maps being degree zero graded $S$-module homomorphisms.
    \par For an $S$-module $M$ and $\p \in \proj(S)$, the homogeneous localization $M_{(\p)}$ is defined to be the localization of $M$ with respect to the multiplicative set containing all the homogeneous elements of $S$ which are not in $\p$. Then $\Supp_S(M)$ is defined to be the collection of all primes $\p \in \proj(S)$ such that $M_{(\p)} \neq 0$.

    Let $X$ denote a quasi-projective scheme over $Spec(A)$ where $A$ is a noetherian ring and $Coh(X)$ the category of coherent sheaves of $\OX$-modules. For $\F \in Coh(X)$, $$\Supp_X(\F) = \{ x \in X \mid \F_x \neq 0 \},$$ which is a closed set in $X$.
    \begin{rmk}\label{rmk:support-of-module-and-tilde}
        Suppose $S = \Oplus_{i \geq 0} S_i$ is finitely generated by $S_1$ as an $S_0$-algebra. For an $S$-module $M$ and $\p \in \proj(S)$, $M_{(\p)} \neq 0$ iff the zeroth homogeneous component of $M_{(\p)}$ is non-zero, since the complement of $\p$ in $S$ always contains a homogeneous element of degree one. Hence, we get $\Supp_S(M) = \Supp_Y(\widetilde{M})$.
    \end{rmk}
    
    We denote by $\C$ an abelian category and recall that $\starmods$ and $Coh(X)$ are special cases of $\C$.

    \subsection{Exact subcategories and their derived categories}
        
    \begin{defn}
    \begin{enumerate}
        \item A full subcategory $\L \subseteq \mathscr C$ is called a \emph{Serre subcategory} of $\mathscr C$ if for any short exact $0 \to M' \to M \to M'' \to 0$  in $\mathscr C$, $M \in \L$ if and only if $M', M'' \in \L$.
        \item A full subcategory $\T \subseteq \mathscr C$ is called a \emph{thick subcategory} of $\mathscr C$ if it is closed under direct summands and has the 2-out-of-3 property, that is, for any short exact sequence in $\C$ in which any two of the objects in the short exact sequence lie in $\T$, so does the third.
    \end{enumerate}
        
    \end{defn}
  In this article, we will primarily focus on the following Serre subcategory: 
  \begin{eg}
  Let $\mathscr C$ be $Coh(X)$ (or $\starmods$). For a closed set $V \subseteq X$ (or $V \subseteq \proj(S)$), the full subcategory $$\L_V = \{ M \in \mathscr C \mid \Supp(M) \subseteq V \}$$ of $\mathscr C$ is a Serre subcategory of $\mathscr C$.
  \end{eg}
  Recall that a full subcategory $\mathscr E \subseteq \C$ is said to be an \emph{exact subcategory} of $\C$ if $\mathscr E$ is closed under extensions. All the exact subcategories we consider in this article are closed under kernels of epimorphisms and closed under direct summands (e.g., thick, or Serre subcategories), and so we assume this hypothesis for the exact category $\SE$ for all definitions/notations ahead.
\begin{defn}
Let $\SE$ be an exact subcategory of $\C$.
    \begin{enumerate}
        \item A chain complex $X_\bullet$ over $\SE$ is said to be \emph{acyclic} if for each $n \in \Z$, there is a short exact sequence 
    $$0 \to Z_n \xrightarrow{a_n} X_n \xrightarrow{b_n} Z_{n-1} \to 0$$ in $\SE$ such that $a_{n} \circ b_{n+1} = \partial_{n+1}^{X}$. Note that this definition is intrinsic and equivalent to $X_\bullet$ being acyclic in $\C$ and $Z_n = Z_n(X_\bullet) \in \SE$ for each $n \in \Z$.
    \item Let $X_\bullet,Y_\bullet$ be chain complexes over $\SE$ and $f \colon X_\bullet \to Y_\bullet$ a map of chain complexes. We say $f$ is a \emph{quasi-isomorphism} if $\cone(f)$ is isomorphic to an acyclic chain complex in the homotopy category of $\SE$.
    \end{enumerate}
\end{defn}

\begin{notationdefinition}
    For $\SE \subseteq \mathscr C$ an exact subcategory, and $\L \subseteq \mathscr C$ a Serre subcategory, denote
    \begin{enumerate}[a)]
            \item $Ch^{b}(\SE)$ (resp. $Ch^{+}(\SE)$): the category of bounded (resp. bounded below) chain complexes over $\SE$.
            \item $Ch^{b}_{\tiny \L}(\SE)$ (resp. $Ch^{+}_{\tiny \L}(\SE)$): the full subcategory of $Ch^{b}(\SE)$ (resp. $Ch^{+}(\SE)$) consisting of chain complexes with all their homologies in $\L$.
            \item $D^b(\SE) = Ch^b(\SE)[\text{Qis}^{-1}]$ : the bounded derived category in which quasi-isomorphisms in $\SE$ are inverted.
            \item $D^b_{\tiny \L}(\SE)$: the full subcategory of $D^b(\SE)$ consisting of chain complexes, with all their homologies in $\L$.
            \end{enumerate}
             For $X_\bullet \in Ch^{+}(\SE)$, 
    \begin{enumerate}[resume,label={(\alph*)}]
            \item $\min_{c}(X_{\bullet}) = \sup\{ n \mid X_i = 0 \text{ for all } i <n  \}$
            \item $\max_{c}(X_{\bullet}) = \inf\{ n \mid X_i = 0 \text{ for all } i >n  \}$
             \item $\min(X_\bullet) = \sup\{ n \mid H_i(X_{\bullet}) = 0 \text{ for all } i <n  \}$ 
            \item $\supph(X_{\bullet}) = \{ n \mid H_n(X_{\bullet}) \neq 0 \}$
            \item $\width(X_\bullet) =  \sup\{ i-j \mid H_i(X_\bullet), H_j(X_\bullet) \neq 0 \}$ if $X_\bullet$ is not acyclic and $\width(X_\bullet) = 0$ if $X_\bullet$ is acyclic
            \item $\Sigma X_\bullet$ denotes the shift of $X_\bullet$, that is, $(\Sigma X_\bullet)_k = X_{k+1}$ 
        \end{enumerate}
\end{notationdefinition}

\begin{rmk}\label{rmk:fully-faithful-thicksub-in-abelian}
    \begin{enumerate}
    \item  For a morphism $f \colon X_\bullet \to Y_\bullet$ in $Ch^{+}(\SE)$, $f$ is a quasi-isomorphism in $Ch^{+}(\SE)$ if and only if $f$ is a quasi-isomorphism in $Ch^{+}(\C)$.
    \item By \cite[Lemma 4.1.16]{HenningKrause}, $D^b(\SE)$ is fully faithful in $D(\SE)$. 
        \item Suppose $\C$ is an abelian category with enough projectives and $\SE$ is an exact subcategory containing the all projective objects of $\C$. Then the natural functor $D^b(\SE) \to D^b(\C)$ is fully faithful. This can be checked, for example, using \cite[Proposition 4.2.15]{HenningKrause} by showing the cofinality condition for the inclusion $\SE \to \C$.
        \item It is also easy to check that $Ch^b_{\tiny \L}(\SE)[\text{Qis}^{-1}] \simeq D^b_{\tiny \L}(\SE)$. Interested readers may refer to \cite[Lemma 1.2.5]{HenningKrause} for a proof.
    \end{enumerate}
\end{rmk}
    \subsection{Graded rings and Koszul complexes}\label{sec:graded-rings-and-koszul-complexes}
       Recall that for a graded module $M$, the graded module $M(n)$ is the same underlying module $M$, but with grading shifted by $n$, i.e., $M(n)_k = M_{k+n}$. For graded $S$-modules $M$ and $N$, recall that the tensor product $M \otimes_S N$ is also a graded $S$-module, with grading given by $(M \otimes_S N)_n = \left\{ \sum_{i} m_i \otimes n_i \mid \deg(m_i)+ \deg(n_i) = n \text{\ for\ all\ } i \right\}$.      
    \par Let $\starf(S)$ denote the full subcategory of $\starmods$ consisting of graded $S$-modules isomorphic to finite direct sums of $S(n)$ where $n \in \Z$. Then $\starf(S)$ is an exact subcategory in $\starmods$. Also, $S(n) \otimes_S S(m) \cong S(m+n)$ for all $m,n \in \Z$ and hence $\starf(S)$ is closed under taking tensor products. Let $\starp(S)$ denote the full subcategory of projective objects in $\starmods$.
    \par Recall that for a commutative ring $R$, $\modr$ denotes the category of finitely generated $R$-modules and morphisms in $\modr$ are $R$-linear maps between $R$-modules and $\mathscr P(R)$ denote the full subcategory of projective $R$-modules in $\modr$.

    \begin{rmk}\label{rmk:mods-has-enough-shifts-of-S}
    \begin{enumerate}[(a)]
        \item Given $P \in \starmods$, $P \in \starp(S)$ if and only if $P \in \mathscr P(S)$. 
        \item Hence (or directly) $\starf(S)$ is a full subcategory of $\starp(S)$.
        \item For any $M \in \starmods$, there exists a surjection $\Oplus_{\text{finite}}S(n_i) \to M$ for some integers $n_i$. Since $S(n_i)$ are projective objects in $\starmods$, the category $\starmods$ has enough projectives.
        \item As a consequence, for $M \in \starmods$, graded projective dimension of $M$ is equal to the (non-graded) projective dimension of $M$. We denote this number by $\pd_S(M)$.
    \end{enumerate}
    \end{rmk}
          \par For finitely generated graded $S$-modules $M$ and $N$, $\Hom_S(M,N)$ has a grading given by $\Hom_S(M,N)_n = \left\{ f \in \Hom_S(M,N) \mid f(M_i) \subseteq N_{i+n}, \text{\ for\ all\ } i \right\}$. Therefore, $\Hom_S(M(-i),N(-j)) \cong \Hom_S(M,N)(i-j)$. Similarly, $\Ext^i_S(M,N)$ is a graded $S$-module. More properties on the category $\starmods$ can be found in \cite{Fossum-Foxby74}.
    
    \par Using the same proof as in \cite[Lemma 5.7.2]{weibelhomological} with the resolutions chosen over $\starf(S)$, we get a Cartan-Eilenberg resolution of any chain complex in $Ch(\starmods)$ with each object in the resolution lying in $\starf(S)$. By taking its total complex, we get the following result.
    \begin{lemma}\label{lem:Cartan-Eilenberg}
        Suppose $X_\bullet \in Ch^{+}(\starmods)$. Then there exists $P_\bullet \in Ch^{+}(\starf(S))$ and a morphism $\pi \colon P_\bullet \to X_\bullet$ in $Ch^{+}(\starmods)$ such that $\pi$ is a quasi-isomorphism.
    \end{lemma}
    
    Let $\tildef= f_1,\dots,f_d \in S$ be a set of homogeneous elements with $\deg(f_i) = n_i$, $\tildef^m = f_1^m,\dots,f_d^m$ and $\stark_{\bullet}(\tildef;S) \in Ch^b(\starf(S))$ the graded Koszul complex on $\tildef$ generated by $e_1 , e_2, \ldots , e_d$  with degrees $n_1, n_2, \ldots , n_d$ respectively. Explicitly
    $$\stark_j(\tildef;S) = \bigoplus_{1 \leq i_1 < \cdots < i_j \leq d} S e_{i_1} \wedge \cdots \wedge e_{i_j}$$ with $\deg(e_{i_1} \wedge \cdots \wedge e_{i_j}) = n_{i_1}+ \cdots +n_{i_j}$ and the usual differentials : $$\partial^{\stark(\tildef;S)}_j(e_{i_1} \wedge \cdots \wedge e_{i_j}) = \sum_{k = 1}^j (-1)^{k+1} f_k e_{i_1} \wedge \cdots \wedge  \widehat{e_{i_k}} \wedge \cdots \wedge e_{i_j}.$$
    Note that each differential in $\stark_{\bullet}(\tildef;S)$ is of degree zero.  For each $M \in \starmods$, define $\stark_\bullet(\tildef;M) \coloneqq \stark_{\bullet}(\tildef;S) \otimes_S M \in Ch^b(\starmods)$. The chain complex obtained by forgetting the grading on $\stark_{\bullet}(\tildef;M)$ is denoted by $K_{\bullet}(\tildef;M)$. The cochain complex $\Hom_S(K_{\bullet}(\tildef;R),M)$ is denoted by $K^{\bullet}(\tildef;M)$.
    \par Given $n \geq m$, we have a graded chain complex map $\kappa^{n,m} \colon \stark_{\bullet}(\tildef^n;S) \to \stark_{\bullet}(\tildef^m;S)$ mapping $e_{i_1} \wedge \cdots \wedge e_{i_j}$ to $(f_{i_1} \cdots f_{i_j})^{n-m} \tilde{e_{i_1}} \wedge \cdots \wedge \tilde{e_{i_j}}$, where $e_1, e_2, \ldots , e_d$ are generators of $\stark_{\bullet}(\tildef^n;S)$ and $\tilde{e_1}, \tilde{e_2}, \ldots , \tilde{e_d}$ are generators of $\stark_{\bullet}(\tildef^m;S)$. Note that this is a degree $0$ map. When $n$ and $m$ are clear from the context, for ease of notation, we denote $\kappa^{n,m}$ by $\kappa$.

    \subsection{Extension and support of coherent sheaves}
  Throughout this subsection, we use the following notation: $Y$ is a noetherian scheme, $X \subseteq Y$ an open subset, and $\iota \colon X \to Y$ the inclusion map. Note that for a coherent sheaf $\F$ over $X$, $\iota_* \F$ is a quasi-coherent sheaf on $Y$. 
\begin{lemma}\label{lem:Hartshorne-extension-of-sheaves}
   Suppose $\F \in Coh(X)$, $h \colon \mathcal Q \to \iota_* \F$ is a morphism of quasi-coherent sheaves on $Y$ with $\mathcal Q \in Coh(Y)$. Then there exists a coherent sheaf $\F'  \subseteq \iota_* \F$ on $Y$ such that $\im(h) \subseteq \F'$, and $\F'|_{X} = \F$.
\end{lemma}
\begin{proof}
    It follows from \cite[Chapter 2, Exercise 5.15]{Hartshorne} that there exists $\F'' \in Coh(Y)$ such that $\F'' \subseteq \iota_* \F$ and $\F''|_{X} = \F$. Let $\F'$ be the sum of subsheaves $\F''$ and $\im(h)$ of $\iota_* \F$. Since $\mathcal Q$ is coherent on $Y$, so is $\im(h)$ and hence $\F' \in Coh(Y)$. Note that $\F'' \subseteq \F' \subseteq \iota_* \F$ and $\F''|_{X} = (\iota_*\F)|_{X}= \F$. Therefore, $\F'|_{X} = \F$. This completes the proof.
\end{proof}

\begin{lemma}\label{lem:extension-ses-sheaves}
    Consider the short exact sequence of coherent sheaves on $X$ $$\CW_\bullet \coloneqq \quad 0 \to \CW_1 \xrightarrow{f} \CW_2 \xrightarrow{g} \CW_3 \to 0$$ with $\CW_1' \subseteq \iota_*\CW_1$ a coherent sheaf on $Y$ whose restriction to $X$ is $\CW_1$. Then there exists a short exact sequence 
    $$\CW'_\bullet \coloneqq \quad 0 \to \CW_1' \xrightarrow{f'} \CW_2' \xrightarrow{g'} \CW_3' \to 0$$ of coherent sheaves on $Y$ such that $\CW_\bullet'$ is a subcomplex of $\iota_*\CW_\bullet$ and ${\CW_\bullet'}|_{X} = \CW_\bullet$.
\end{lemma}
\begin{proof}
     We have $\CW_2 \in Coh(X)$ and $\iota_* f |_{\CW_1'} \colon \CW_1' \to \iota_* \CW_2$ a morphism of quasi-coherent sheaves on $Y$ with $\CW_1' \in Coh(Y)$. By applying \Cref{lem:Hartshorne-extension-of-sheaves}, there exists a coherent sheaf $\CW_2' \subseteq \iota_*\CW_2$ such that ${\CW_2'|}_{X} = \CW_2$ and $\im \left( \iota_* f |_{\CW_1'} \right) \subseteq \CW_2'$. 
     Since $\iota_* f$ is injective, its restriction $f' \colon \CW_1' \to \CW_2'$ is also injective. Let $(\CW_3',g')$ be the cokernel of $f'$. Since the restriction of sheaves from $Y$ to $X$ is an exact functor, it is easy to see that ${\CW_\bullet'}|_{X} = \CW_\bullet$. 
\end{proof}

For a topological space $Y$ and $A \subseteq Y$, the smallest closed set containing $A$ in $Y$ is denoted by $\overline{A}^Y$.
\par The following lemma establishes the relation between the support of a coherent sheaf on an open set and the support of its coherent extension to the ambient scheme.
\begin{lemma}\label{lem:support-under-extn}
    Let $\CW$ be a coherent sheaf on $X$ and $\CW' \subseteq \iota_*\CW$ a coherent sheaf on $Y$ such that $\CW'|_{X} = \CW$, then $\Supp_Y(\CW') = \overline{\Supp_X(\CW)}^Y$.
\end{lemma}
\begin{proof}
    Clearly $\Supp_X(\CW) \subseteq \Supp_Y(\CW')$. Since $\CW'$ is coherent, its support is closed in $Y$. Thus $\overline{\Supp_X(\CW)}^Y \subseteq \Supp_Y(\CW')$. To show the other containment, suppose $x \notin \overline{\Supp_X(\CW)}^Y$. Then there exists an open set $U \subseteq Y$ such that $x \in U$ and $U \cap \overline{\Supp_X(\CW)}^Y = \emptyset$. Now $\iota_*(\CW)(U) = \CW(U \cap X) = 0$. Since $\CW'$ is a subsheaf of $\iota_*(\CW)$, it follows that $\CW'(U) = 0$ and hence $x \notin \Supp_Y(\CW')$.
\end{proof}
\begin{lemma}\label{lem:extension-of-sheaves}
Let $Y$ be a noetherian scheme and $X \subseteq Y$ an open subset. Let $\F_\bullet$ be a bounded chain complex of coherent sheaves on $X$. Then there exists a bounded chain complex $\G_\bullet$ of coherent sheaves on $Y$ such that $\G_{\bullet}|_{X} = \F_\bullet$ and the closure of $\Supp_{X}(\H_i(\F_\bullet))$ in $Y$ is equal to $\Supp_Y(\H_i(\G_\bullet))$ for each $i$.
\end{lemma}
\begin{proof}
    Let $\F_\bullet \coloneqq 0 \to \F_m \xrightarrow{\phi_m} \F_{m-1} \to \cdots \to \F_{k+1} \xrightarrow{\phi_{k+1}} \F_k \to 0$ for some $m, k \in \mathbb Z$ where each $\F_i$ is a coherent sheaf on $X$. Let us denote the sheaves $\CZ_i \coloneqq \ker(\phi_i \colon \F_i \to \F_{i-1})$, $\CB_i \coloneqq \im(\phi_{i+1} \colon \F_{i+1} \to \F_i)$ and $\H_i \coloneqq \CZ_i/\CB_i$. Since $X$ is also noetherian, $\CZ_i, \CB_i$ and $\H_i$ are coherent sheaves on $X$. We break $\F_\bullet$ into short exact sequences
    \begin{equation}\label{eqn:ses1}
        0 \to \CB_i \to \CZ_i \to \H_i \to 0 \tag{$A_i$}
    \end{equation} 
    \begin{equation}\label{eqn:ses2}
        0 \to \CZ_i \to \F_i \to \CB_{i-1} \to 0 \tag{$B_i$}
    \end{equation}
    \par In order to lift $\F_\bullet$ to a chain complex $\G_\bullet$ such that $\G_{\bullet}|_{X} = \F_\bullet$, it is sufficient to lift the short exact sequences $A_i$ and $B_i$ to short exact sequences $A_i'$ and $B_i'$ respectively, of coherent sheaves on $Y$ such that $A'_{i}|_X = A_i$ and $B'_{i}|_X = B_i$ for each $i$. Moreover we will construct $A_i' \subseteq \iota_*A_i$ and $B_i' \subseteq \iota_*B_i$ where $\iota\colon X \to Y$ is the inclusion map. Then we will show that this yields the closure of $\Supp_{X}(\H_i(\F_\bullet))$ in $Y$ is equal to $\Supp_Y(\H_i(\G_\bullet))$.
    \par By (reverse) induction on $i$, we lift $A_i$ and $B_i$. The base case is when $i = m$. Consider the short exact sequence $A_m$. Here, $\CB_m = 0$ and hence $\iota_* \CB_m = 0$ is coherent on $Y$. Applying \Cref{lem:extension-ses-sheaves} to $A_m$, we get a short exact sequence $A_m' \subseteq \iota_*A_m$ on $Y$, $$A_m' \coloneqq \quad 0 \to 0 \to \CZ_m' \to \H_m' \to 0$$ whose restriction to $X$ is $A_m$. Now consider the short exact sequence $B_m$. We already have a lift of $\CZ_m$ from $A_m$. Again applying \Cref{lem:extension-ses-sheaves} for $B_m$, we get the exact sequence $B_m' \subseteq \iota_*B_m$ on $Y$, $$B_m' \coloneqq \quad 0 \to \CZ_m' \to \F_m' \to \CB_{m-1}' \to 0$$ such that ${B_m'|}_{X} = B_m$.
    \par Let us assume that for all $j>i$, we have short exact sequences $A_j' \subseteq \iota_*A_j$ and $B_j' \subseteq \iota_*B_j$ of coherent sheaves on $Y$ such that ${A_j'|}_{X} = A_j$ and ${B_j'|}_{X} = B_j$. Consider the short exact sequence $A_i$. Since $B_{i+1}$ is already lifted to $B_{i+1}'$, we have a coherent sheaf $\CB_i' \subseteq \iota_*\CB_i$ whose restriction to $X$ is $\CB_i$. Now apply \Cref{lem:extension-ses-sheaves} to get $A_i'$ as required. From $A_i'$, we get $\CZ_i' \subseteq \iota_*\CZ_i$ such that ${\CZ_i'|}_{X} = \CZ_i$. Again applying \Cref{lem:extension-ses-sheaves}, we get $B_i'$ as required. Hence, the induction step is complete. Now the short exact sequences $A_i'$ and $B_i'$ together yield a chain complex $\G_\bullet \coloneqq \F_\bullet'$ whose restriction to $X$ is $\F_\bullet$.
    \par Since $\H_i(\G_\bullet) = \H_i' \subseteq \iota_*\H_i(\F_\bullet)$ by construction, and ${\H_i(\G_\bullet)|}_{X} = \H_i(\F_\bullet)$, using \Cref{lem:support-under-extn}, we get that the closure of $\Supp_{X}(\H_i(\F_\bullet))$ in $Y$ is equal to $\Supp_Y(\H_i(\G_\bullet))$ for each $i$.
\end{proof}
 The next lemma compares the support of a chain complex of coherent sheaves and the support upon applying the functor $\Gamma_*(-) \colon Coh(Y) \to \starmods$ for a projective scheme $Y$. This is a direct consequence of \cite[Chapter 2, Proposition 5.15]{Hartshorne}.
\begin{lemma}\label{lem:del-comp-gamma-is-identity}
    Let $S = \Oplus_{i \geq 0} S_i$ be a graded ring which is finitely generated by $S_1$ as an $S_0$-algebra, $Y = \proj(S)$ and $\F_\bullet \in Ch^b(Coh(Y))$. Then $$\Supp_Y(\H_i(\F_\bullet)) = \Supp_S(H_i(\Gamma_*(\F_\bullet))).$$
\end{lemma}
\section{Graded Tate resolutions, Koszul complexes and local cohomology}\label{sec:graded Tate etc.}
    The following lemma shows the existence of graded chain complex maps, when chain complex maps which are not necessarily graded exist.
     \begin{lemma}\label{rmk:existence-of-graded-map}
     Let $F_\bullet, G_\bullet \in Ch(\starmods)$ and a morphism $f' \colon F_\bullet \to G_\bullet$ in $Ch(\mods)$.  Then there exists a morphism $f \colon F_\bullet \to G_\bullet$ in $Ch(\starmods)$ such that for each $i$ and $x \in F_i$ a homogeneous element of degree $d$, if $f_i'(x)$ is homogeneous of degree $d$, then $f_i(x) = f_i'(x)$. 
    \end{lemma}
    \begin{proof}
        For each $i$ and $x \in F_i$ a homogeneous element of degree $d$, let $f_i(x) = \pi_d^i \circ f_i'(x)$ where $\pi_d^i$ is the projection of $G_i$ to its $d$-th homogeneous component. Extend $f_i$ linearly to $F_i$.  Clearly, $f_i$ is a morphism in $\starmods$ and $f_i(x) = f_i'(x)$ whenever $x$ and $f_i'(x)$ is homogeneous of degree $d$. Note that
    \begin{align*}
    f_{i-1}\partial^F_i(x) & = \pi_d^{i-1} f_{i-1}'
    \partial^F_i(x) & \quad (\text{since } \partial^F_i \text{ is graded of degree } 0 ) \\
    & = \pi_d^{i-1} \partial_i^G f_i'(x) & \quad (\text{since } f_i' \text{ is a chain complex morphism})  \\
    & = \partial_i^G \pi_d^i f_i'(x) & \quad (\text{since } \partial^G_i \text{ is graded of degree } 0) \\
    & = \partial_i^G f_i(x).
    \end{align*}
    This shows that $f_{i-1}\partial^F_i = \partial_i^G f_i$ and hence $f$ is a morphism in $Ch(\starmods)$.
    \end{proof}    
    \par The following lemma provides a graded version of \cite[Lemma 3.1]{SandersSane}. The original statement is an ungraded dual version due to \cite[Proposition 23]{Foxby-Halvorsen}.
   \begin{lemma}\label{lem:graded-foxby}
        Let $I \subseteq S$ be an ideal generated by homogeneous elements $\tildes = s_1,\dots,s_d$, $P_\bullet \in Ch^+(\starf(S))$
        such that $\min(P_\bullet) = 0$ and $H_i(P_\bullet)$ is supported on $V(I)$ for all $i$. Then there exist $n \in \N$ and a graded chain complex map $$\delta \colon \stark_\bullet(\tildes^n;P_0) \to P_\bullet$$ with $\delta_0\colon S \otimes_S P_0 \to P_0$ the natural graded isomorphism. 
    \end{lemma}
    \begin{proof}
    Forgetting the grading on $P_\bullet$, it follows from \cite[Proof of Lemma 3.1]{SandersSane} that there exists $n \in \N$ and a chain complex map in $Ch^{+}(\mods)$ $$\delta' \colon K_\bullet(\tildes^n;P_0) \to P_\bullet$$ such that $\delta'_0\colon S \otimes_S P_0 \to P_0$ is the natural isomorphism.
    Now applying \Cref{rmk:existence-of-graded-map} to $f' = \delta'$,
    we get the required result.
    \end{proof}
Recall that for a chain complex $F_\bullet \in Ch(\starmods)$, for each $i$, $Z_i(F_\bullet)$ and $B_i(F_\bullet)$ denote $\ker(F_i \to F_{i-1})$ and $\im(F_{i+1} \to F_i)$ respectively, and these are also graded $S$-modules.
    Now we will define the graded analogue of the Tate resolution. 
    \begin{defn}\label{defn:graded-version-of-Tate}
        Let $\tildef = f_1 \ldots, f_d \in S$ be a set of homogeneous elements. Define
\[
\start_i(\tildef) = \begin{cases}
            0 & \textrm{ if }~  i < 0 \\
            \stark_i(\tildef;S) & \textrm{ if }~ i = 0, 1 \\
            \stark_i(\tildef;S) \oplus \left(\Oplus_{k = 1}^{t_i} S(-l_{ik}) \right) & \textrm{ if }~  i \geq 2
        \end{cases}
\]
where the differentials $\partial_{i-1}^{\tiny \start(\tildef)}$, $l_{ik}$ and $t_i$ are inductively defined as follows : $\partial_1^{\tiny \start_{\bullet}(\tildef)} \coloneqq \partial_1^{\tiny  \stark_{\bullet}(\tildef;S)}$. Choose a set $X_1 = \{a_{21},\dots,a_{2t_2}\}$ of homogeneous elements generating $Z_1(\start_{\bullet}(\tildef))$ with $\deg(a_{2k}) = l_{2k}$ for all $1 \leq k \leq t_2$. 
\par Suppose for each $1 \leq j \leq r-1$, the differentials $\partial_j^{\tiny \start(\tildef)}$ have been defined and sets $X_j = \{a_{(j+1)1},\dots,a_{(j+1)t_{j+1}} \}$ of homogeneous elements generating $Z_j(\start_{\bullet}(\tildef))$ with $\deg(a_{(j+1)k}) = l_{(j+1)k}$ have been chosen for all $1 \leq k \leq t_{j+1}$. Define $\partial_r^{\tiny \start(\tildef)}$ by
$\partial_r^{\tiny \start(\tildef)}|_{\stark_r(\tildef;S)} = \partial_r^{\stark(\tildef;S)}$ and $\partial_r^{\tiny \start(\tildef)}$ maps the generator of $S(-l_{rk})$ in degree $l_{rk}$ to $a_{rk} \in X_{r-1} \subseteq Z_{r-1}(\start_{\bullet}(\tildef))$. Choose a set $X_r = \{a_{(r+1)1},\dots,a_{(r+1)t_{r+1}} \}$ of homogeneous elements generating $Z_r(\start_{\bullet}(\tildef))$ with  $\deg(a_{(r+1)k}) = l_{(r+1)k}$ for all $1 \leq k \leq t_{r+1}$, completing the induction step.
\par We call $\start_{\bullet}(\tildef)$ the \emph{graded Tate resolution} on $\tildef$. 
    \end{defn}
\begin{rmk}
    By construction, $\stark_{\bullet}(\tildef;S)$ is a subcomplex of $\start_{\bullet}(\tildef)$. Also, note that $\start_{\bullet}(\tildef)$ is a graded free resolution of $S/(\tildef)$.
\end{rmk} 
    \begin{lemma}\label{lem:graded-tate-to-koszul}
        Given a set of homogeneous elements $\tildes = s_1,\dots,s_d \in S$ and $r \in \N$, there exists $u(r) \geq r$ and a graded chain complex map $\varphi \colon \start_{\bullet}(\tildes^{u(r)}) \to \stark_{\bullet}(\tildes^r;S)$ such that $\varphi|_{\stark_{\bullet}(\tildes^{u(r)};S)} = \kappa^{u(r),r}$. 
    \end{lemma}
    \begin{proof}
        Let $T_{\bullet}(\tildes^{u(r)})$ be the Tate resolution on $\tildes^{u(r)}$, obtained by forgetting the grading on $\start_{\bullet}(\tildes^{u(r)})$. It follows from \cite[Theorem 3.5]{GanapathySane} that there exists $u(r) \geq r$ and a chain complex map $\varphi'\colon T_{\bullet}(\tildes^{u(r)}) \to K_{\bullet}(\tildes^r;S)$ such that $\varphi'|_{K_{\bullet}(\tildes^{u(r)};S)} = \kappa^{u(r),r}$. By applying \Cref{rmk:existence-of-graded-map} to $f' = \varphi'$ and observing that $\varphi'|_{K_{\bullet}(\tildes^{u(r)};S)}$ is a degree zero graded map, we get the required result.
    \end{proof}

\begin{lemma}\label{lem:Tate-to-bdd-above-complex}
    Let $I \subseteq S$ be a homogeneous ideal and $\L = \L_{V(I)}$. Suppose that $P_\bullet \in Ch^{+}_{\tiny \L}(\starf(S))$ such that $\min(P_\bullet) = 0$ and $\delta \colon P_0 \to Q$ is a morphism with $Q \in \L$. Then there exist a homogeneous ideal $J \subseteq I$ with $\sqrt{J} = \sqrt{I}$, $U_\bullet \in Ch^+_{\L}(\starf(S))$ and a graded chain complex map $\psi \colon U_\bullet \to P_\bullet$ such that
         \begin{enumerate}[(i)]
            \item $\min_c(U_\bullet) = 0$ with $U_0 = S \otimes_S P_0$.
            \item $\supph(U_\bullet) = \{ 0 \}$.
            \item there is a graded isomorphism $H_0(U_\bullet) \cong S/J \otimes_S P_0$ .
            \item $H_0(\psi) \colon H_0(U_\bullet) \to H_0(P_\bullet)$ is an epimorphism.
            \item $\delta \circ \psi_0$ factors through $H_0(U_\bullet)$.
        \end{enumerate}
\end{lemma}
\begin{proof}
    Let $I$ be generated by homogeneous elements $s_1,\dots,s_d$. Since $Q \in \L$, there exists $n \in \N$ such that $(\tildes^n) \subseteq \Ann_S(Q)$. As a direct consequence of \Cref{lem:graded-foxby} and \Cref{lem:graded-tate-to-koszul}, we get a map $\psi \colon \start_\bullet(\tildes^u) \otimes P_0 \to P_\bullet$ for all $u \gg 0$. By choosing $u \geq n$ and taking $U_\bullet = \start_\bullet(\tildes^u) \otimes P_0$ and $J = (\tildes^u)$, we can easily see that the properties \textit{(i)} to \textit{(v)} are satisfied. 
\end{proof}

Now we will discuss an application of \Cref{lem:graded-tate-to-koszul} related to the graded \v{C}ech complex and graded local cohomology.
 
For a set of homogeneous elements $\tildes = s_1,\dots,s_d \in S$ and a graded $S$-module $M$, let $\check{C}_{\bullet}(\tildes;M)$ denote the graded \v{C}ech complex on $\tildes$ with coefficients in $M$. 
\begin{thm}\label{thm:local-cohomology}
    Let $S$ be a commutative noetherian graded ring, $\tildes = s_1,\dots,s_d$  homogeneous elements in $S$, and $M$ a graded $S$-module. 
    Then, there exists an infinite set $E \subseteq \N$ such that $$\check{C}^{\bullet}(\tildes;M) \cong \varinjlim_E \Hom_S(\start_{\bullet}(\tildes^n),M)$$ for some choice of graded chain complex maps $\start_{\bullet}(\tildes^{n_2}) \to \start_{\bullet}(\tildes^{n_1})$, which lift the natural surjections $S/(\tildes^{n_2}) \to S/(\tildes^{n_1})$ for all $n_1,n_2 \in E$ with $n_2>n_1$.
\end{thm}
\begin{proof}
    Recall the function $u \colon \N \to \N$ in \Cref{lem:graded-tate-to-koszul}. Let $E = \{u_t \mid t \in \N \}$ where $u_t = (\underbrace{u \circ u \circ \cdots \circ u}_{t-\text{times}})(1) \in \N$ and accordingly the map in \Cref{lem:graded-tate-to-koszul} is denoted by $\varphi_{t} \colon \start_\bullet(\tildes^{u_t}) \to \stark_\bullet(\tildes^{u_{t-1}})$. Let $i_{t} \colon \stark_{\bullet}(\tildes^{u_{t}};S) \to \start_{\bullet}(\tildes^{u_{t}})$ be the inclusion map. Applying \Cref{lem:graded-tate-to-koszul} with the above notations repeatedly, we get that for any $t \in \N$, the following diagram is commutative
    \begin{equation}\label{diag:tate-to-koszul}
\xymatrix{ \stark_{\bullet}(\tildes^{u_{t+2}};R) \ar@{^{(}->}[r]^{i_{t+2}} \ar[d]_{\kappa}& \start_{\bullet}(\tildes^{u_{t+2}}) \ar@{-->}[ld]_{\varphi_{t+2}} \ar[d]^{i_{t+1} \circ \varphi_{t+2} } \ar[r]^{\varphi_{t+2}} & \stark_{\bullet}(\tildes^{u_{t+1}};R) \ar@{^{(}->}[r]^{i_{t+1}}  \ar[d]^{\kappa} & \start_{\bullet}(\tildes^{u_{t+1}}) \ar@{-->}[ld]^{\varphi_{t+1}} \ar[d]^{i_t \circ \varphi_{t+1} } \\  \stark_{\bullet}(\tildes^{u_{t+1}};R) \ar@{^{(}->}[r]^{i_{t+1}}  & \start_{\bullet}(\tildes^{u_{t+1}}) \ar[r]^{\varphi_{t+1}} & \stark_{\bullet}(\tildes^{u_t};R)\ar@{^{(}->}[r]^{i_{t}} & \start_{\bullet}(\tildes^{u_t}) \\}
\end{equation}
Note that $\varphi_t \circ i_t$ is the map $\kappa^{u_t,u_{t-1}}$. Consider the directed system induced by the maps in the above diagram upon applying the functor $\Hom_S(-,M)$:
$$\{ \Hom_S(\start_{\bullet}(\tildes^{u_t}),M), i_{t-1} \circ \varphi_{t}\}_{t \in \N}.$$ Taking the limits and using the commutative diagram in \eqref{diag:tate-to-koszul}, we get $$\varinjlim_{t \in \N} \stark^{\bullet}(\tildes^n;M) \cong \varinjlim_{t \in \N} \stark^{\bullet}(\tildes^{u_t};M) \cong \varinjlim_{t \in \N} \Hom_S(\start_{\bullet}(\tildes^{u_t}),M)$$
Since $\check{C}^{\bullet}(\tildes;M) \cong \varinjlim_{n \in \N} \stark^{\bullet}(\tildes^n;M)$, the result follows. 
\end{proof}
\begin{rmk}
Recall that the $i$-th local cohomology module $H^i_{(\tildef)}(M)$ can be computed in two ways:
\begin{equation}\label{eqn:two-descriptions-of-local-cohomology}
H^i_{(\tildef)}(M) \coloneqq \varinjlim \Ext^i_S(S/(\tildef^n),M) \cong \varinjlim H^i(\stark^{\bullet}(\tildef^n;M)).    
\end{equation}
 Whenever $S$ is a graded ring, all the modules and the maps in the direct limits are graded, and hence the grading transfers to the local cohomology module as well. 
\par Taking cohomology of the \v{C}ech complex, using the description in \Cref{thm:local-cohomology}, we obtain a new proof of the well-known fact that the two definitions of the local cohomology module in \eqref{eqn:two-descriptions-of-local-cohomology} are equivalent. Also, it is clear from the diagram \eqref{diag:tate-to-koszul} that given any $n \in \mathbb N$, there exists a graded map \begin{equation}\label{eqn:map-from-koszul-cohomology-to-tate-cohomology}
    H^i(\stark^{\bullet}(\tildef^{u(n)};M)) \xrightarrow[]{\Hom(\varphi,M)} \Ext^i_S(S/(\tildef^n),M).
\end{equation} 
\end{rmk}
Now we end the section with a version of \Cref{lem:Tate-to-bdd-above-complex} for quasi-projective schemes. We will use this in \Cref{sec: der equiv quasi proj} later. Recall the notations: for a scheme $X$, let $Coh(X)$ denote the category of coherent sheaves of $\OX$-modules,  $\V(X)$  the category of locally free sheaves on $X$ of finite rank, and $\L$ a Serre subcategory of $Coh(X)$.
\begin{lemma}\label{lem:foxby-morphism-for-schemes}
    Let $X$ be a quasi-projective scheme over a noetherian affine scheme and $\L = \L_{V}$ for a closed set $V \subseteq X$. Suppose that $\G_\bullet \in Ch^b_{\tiny \L}(Coh(X))$ with $\min(\G_\bullet) = m$ and $\lambda \colon \G_m \to \mathcal Q$ be a morphism with $Q \in \L$. Then there exists a chain complex $\CU_\bullet \in Ch^{+}_{\tiny \L}(\V(X))$ and a chain complex map $\beta\colon \CU_\bullet \to \G_\bullet$ such that
    \begin{enumerate}
        \item $\min_{c}(\CU_\bullet) = m$.
        \item $\supph(\CU_\bullet) = \{ m \}$.
        \item $\H_m(\CU_\bullet) = (\OX/\I) \otimes \OX(n)$ for some ideal sheaf $\I$ defining a closed subscheme structure on $V$ and $n \in \N$.
        \item $\H_m(\beta)\colon \H_m(\CU_\bullet) \to \H_m(\G_\bullet)$ is an epimorphism.
        \item $\lambda \circ \beta_m$ factors through $\H_m(\CU_\bullet)$.
    \end{enumerate} 
    \end{lemma}
\begin{proof}
    Let $Y$ be a projective scheme over a noetherian affine scheme with $X$ an open subset in $Y$ and $\iota \colon X \to Y$ the inclusion map. Then $Y = \proj(S)$ where $S = \Oplus_{i \geq 0} S_i$ is a graded ring and is finitely generated by $S_1$ as an $S_0$-algebra. 
    \par Without loss of generality, we may assume $\min(\G_\bullet) = 0$. Extend $\G_\bullet$ to a bounded chain complex of coherent sheaves $\F_\bullet$ over $Y$ such that $\Supp_Y(\F_\bullet) = \overline{\Supp_X(\G_\bullet)}^Y$ as in \Cref{lem:extension-of-sheaves}.
    Thus $\F_\bullet \in Ch^b(Coh(Y))$ and $\H_i(\F_\bullet)$ is supported on $\overline{V}^Y$ for each $i$. 
    \par The given map $\lambda \colon \G_0 \to \mathcal Q$ induces the map $\iota_*\lambda \colon \iota_* \G_0 \to \iota_*\mathcal Q$ of quasi-coherent sheaves on $Y$. Now applying \Cref{lem:Hartshorne-extension-of-sheaves} to $\mathcal Q \in Coh(X)$ and $\iota_*\lambda|_{\mathcal F_0} \colon \F_0 \to \iota_*\mathcal Q$, there exists a coherent sheaf $\mathcal Q' \subseteq \iota_*\mathcal Q$ on $Y$ such that $\mathcal Q'|_{X} = \mathcal Q$ and $\im \left(\iota_*\lambda|_{\mathcal F_0} \right)\subseteq \mathcal Q'$. Thus, $\iota_*\lambda$ restricts to a morphism $\lambda' \colon \F_0 \to \mathcal Q'$ of coherent sheaves on $Y$, which when further restricted to $X$ is the map $\lambda$. 
    \par Note that $\overline{V}^Y = V(I)$ for some homogeneous ideal $I \subseteq S$. Let $Q' = \Gamma_*(\mathcal Q')$ and $F_\bullet \coloneqq \Gamma_*(\F_\bullet) \in Ch^b(\starmods)$. By \Cref{lem:del-comp-gamma-is-identity}, $H_i(F_\bullet)$ is supported on $V(I)$ for all $i$. It follows from \Cref{lem:Cartan-Eilenberg} that there exists $P_\bullet \in Ch^{+}(\starf(S))$ and a graded chain complex map $\pi \colon P_\bullet \to F_\bullet$ such that $\pi$ is a quasi-isomorphism and hence $\min(P_\bullet) = \min(F_\bullet) = 0$.  By \Cref{rmk:support-of-module-and-tilde}, $\Supp_S(Q') = \Supp_Y(\mathcal Q')$. Since $\Supp_Y(\mathcal Q') \subseteq V(I)$, we get $I \subseteq \sqrt{\Ann_S(Q')}$. Let $s_1,\dots,s_d$ be a set of homogeneous generators of $I$. Thus for $n \gg 0$, $(s_1^n,\ldots,s_d^n) \subseteq \Ann_S(Q')$. Applying Lemma \ref{lem:Tate-to-bdd-above-complex} to the chain complex $P_\bullet$ and the map $\Gamma_*(\lambda') \circ \pi_0 \colon P_0 \to Q'$, there exist a homogeneous ideal $J \subseteq I$ with $\sqrt{J} = \sqrt{I}$, $U_\bullet \in Ch^{+}_{V(I)}(\starf(S))$ and a graded chain complex map $\psi\colon U_\bullet \to P_\bullet$ such that 
    \begin{enumerate}[(i)]
            \item $\min_c(U_\bullet) = 0$ with $U_0 = S \otimes_S P_0$. 
            \item $\supph(U_\bullet) = \{ 0\}$.
            \item there is a graded isomorphism $H_0(U_\bullet) \cong 
                S/J \otimes_S P_0$.
            \item $H_0(\psi) \colon H_0(U_\bullet) \to H_0(P_\bullet)$ is an epimorphism.
            \item  $\Gamma_*(\lambda') \circ \pi_0 \circ \psi_0$ factors through $H_0(U_\bullet)$.
        \end{enumerate}
Define $\gamma \coloneqq \pi \circ \psi \colon U_\bullet \to F_\bullet$ , $\L' = \L_{\overline{V}^Y}$ and $\CU_\bullet' = 
\widetilde{U_\bullet}$. Applying $\widetilde{(-)}$ followed by the natural isomorphism $\widetilde{F_\bullet} \to \F_\bullet$ yields a chain complex map $\widetilde{\gamma}\colon \CU_\bullet' \to \F_\bullet$ in $Ch^+_{\L'}(\V(Y))$. Since $\widetilde{(-)}$ is an exact functor, it follows that :
    \begin{enumerate}[(a)]
        \item $\min_c(\CU_\bullet') = 0$.
        \item $\supph(\CU_\bullet') = \{ 0\}$.
        \item $\H_0(\CU_\bullet') \cong \widetilde{S/J \otimes_S P_0}$, which is supported on $V(I)$ (using \Cref{rmk:support-of-module-and-tilde}).
        \item $\H_0(\widetilde{\gamma})\colon \H_0(\CU_\bullet') \to \H_0(\F_\bullet)$ is an epimorphism.
        \item $\lambda' \circ \tilde{\gamma}_0$ factors through $\H_0(\CU_\bullet')$.
    \end{enumerate}
    \par Since $V(I) \cap X = V$, restricting $\CU_\bullet'$ to $X$ is a chain complex of coherent sheaves on $X$ all of whose homologies are supported on $V$. Since a morphism is an epimorphism iff it is so for each stalk, $\H_0(\widetilde{\gamma})$ is an epimorphism on $Coh(Y)$, which implies that its restriction to $X$ is also an epimorphism. Letting $\beta\coloneqq \widetilde{\gamma|}_{X}$ and $\CU_\bullet \coloneqq \CU_{\bullet}'|_X$, we get $\CU_\bullet \in Ch^{+}_{\tiny \L}(\V(X))$ with $\beta \colon \CU_\bullet \to \G_\bullet$ satisfying the required properties. 
    \end{proof}

\section{The SR-property and a derived equivalence}\label{sec:derived-equivalence-for-abelian-category}

We recall the notations which will be used in this section. Let $\mathscr C$ be an abelian category, $\CA$ a resolving subcategory of $\mathscr C$, $\T$ a thick subcategory containing $\CA$, and $\L$ a Serre subcategory of $\mathscr C$. In this section we show that under certain conditions on $\T$ and $\L$, the functor $\iota \colon D^b(\T \cap \L) \to D^b_{\tiny \L}(\T)$ is an equivalence of categories. 
\begin{defn}\label{defn:SR-property-for-an-abelian-category}
    A Serre subcategory $\L \subseteq \mathscr C$ is said to have the \emph{SR property} with respect to a thick subcategory $\T$ of $\mathscr C$ if for each $X_\bullet \in Ch^b_{\tiny \L}(\C)$ which is not exact with $\min(X_\bullet) = m$ and a morphism $f \colon X_m \to Q$ with $Q \in \L$, there exist $T_\bullet \in Ch^b_{\tiny \L}(\T)$ and a chain complex map $\alpha \colon T_\bullet \to X_\bullet$ such that \begin{enumerate}
        \item  $\min_c(T_\bullet) = m$.
        \item $\supph(T_\bullet) = \{m\}$.
        \item $H_m(\alpha) \colon H_m(T_\bullet) \to H_m(X_\bullet)$ is an epimorphism in $\C$.
        \item $f \circ \alpha_m$ factors through $H_m(T_\bullet)$.
    \end{enumerate}
\end{defn}
We call such a pair $(T_\bullet, \alpha)$ a \textit{strong reducer} of $(X_\bullet,f)$. This generalizes the definition of strong reducers in \cite[Definition 4.1]{GanapathySane}. 
Since many arguments in \Cref{lem:morphisms-in-derived-category-abelian-category}, \Cref{lem:fullness-of-the-map} and \Cref{thm:derived-equivalence-main-abelian-category} parallel those in \cite[Sections 3 and 4]{SandersSane}, we give detailed proofs only for new points and brief outlines otherwise.
\begin{lemma}\label{lem:morphisms-in-derived-category-abelian-category}
     Let ${X}_\bullet, {Y}_\bullet \in Ch^b(\T \cap \L)$ and ${X}_\bullet \xrightarrow{f} {Y}_\bullet$ a morphism in $D^b_{\tiny \L}(\T)$ with $\min_c({X}_{\bullet}), \min_c({Y}_\bullet) \geq m$ and $\min({X}_\bullet \oplus {Y}_\bullet) = m$. Suppose $\L$ has the SR property with respect to $\T$. Then there exist chain complexes ${M}_{\bullet}^{ X}, {M}_{\bullet}^{ Y}$ in $Ch^b(\T \cap \L)$ and chain complex maps ${M}_{\bullet}^{ X} \xrightarrow{\rho} {M}_{\bullet}^{ Y}$, ${M}_{\bullet}^{ X} \xrightarrow{\beta^X} {X}_{\bullet}$ and ${M}_{\bullet}^{{Y}} \xrightarrow{\beta^Y} {Y}_{\bullet}$ such that
    \begin{itemize}
        \item ${M}_\bullet^{ X} , {M}_\bullet^{ Y} \in Ch^b(\T \cap \L)$ with ${M}_i^{ X} = {M}_i^{ Y} = 0$ for all $i \neq m$.
        \item there is a commutative square in $\DbLT$: 
        $$\xymatrix{ {M}_\bullet^{ X} \ar[r]^{\beta^X} \ar[d]^{\rho} & {X}_\bullet \ar[d]^{f\qquad .}\\ {M}_\bullet^{ Y} \ar[r]^{\beta^Y} & {Y}_\bullet }$$
        \item $H_m(\beta^X)$ and $H_m(\beta^Y)$ are epimorphisms in $\C$. 
    \end{itemize}
\end{lemma}
\begin{proof}
    Define the chain complex ${M}_\bullet^{ Y}$ to be $\Sigma^m  Y_m$ and $\beta^Y\colon {M}_\bullet^{ Y} \to {Y}_\bullet$ the inclusion map. Then ${M}_\bullet^{ Y} \in Ch^b(\T \cap \L)$ and $\beta^Y$ is a chain complex map such that $H_m(\beta^Y)$ is an epimorphism in $\C$. If $H_m(X_\bullet) = 0$, we may choose $M_\bullet^X = 0$ and we are done. Thus, we assume $H_m(X_\bullet) \neq 0$. \\
Let $f$ be given by a roof diagram ${X}_\bullet \xleftarrow{q} {Q}_\bullet \xrightarrow{g} {Y}_\bullet$ where $q$ is a quasi-isomorphism. We may assume $\min_c({Q}_\bullet) = m$. Let $Q_\bullet'$ be the pullback of $\beta^Y$ and $g$ in $Ch^b(\C)$: 
$$\xymatrix{ Q_\bullet' \ar@{..>}[r]^{\nu} \ar@{..>}[d]^{\mu} & Q_\bullet\ar[d]^{g  \qquad .} \\ M_\bullet^Y \ar[r]^{\beta^Y} & Y_\bullet }$$ Using the exact sequence $0 \to Q_\bullet' \to Q_\bullet \oplus M_\bullet^Y \to Y_\bullet$ and the fact that $\L$ is a Serre subcategory, it is easy to check that $Q_\bullet' \in Ch^{b}_{\tiny \L}(\C)$. Note that $H_m(Q_\bullet') \neq 0$ because $H_m(Q_\bullet) \simeq H_m(X_\bullet) \neq 0$. Since $\L$ has the SR property with respect to $\T$, $(Q'_\bullet,q_m \circ \nu_m)$ has a strong reducer $(T_\bullet,\alpha)$. Therefore, we have the following commutative diagram in $D^{b}_{\tiny \L}(\C)$: 
$$\xymatrix{ T_\bullet \ar[r]^{\alpha} & Q_\bullet' \ar[r]^{\lambda} \ar[d]^{\mu} & X_\bullet\ar[d]^{f \qquad .} \\ & M_\bullet^Y \ar[r]^{\beta^Y} & Y_\bullet }$$
Define $M_\bullet^X$ to be the chain complex $\Sigma^m H_m(T_\bullet) \in   Ch^b(\T \cap \L)$. By property (4) of the strong reducer, we get a chain complex map $\beta^X \colon M_\bullet^X \to X_\bullet$.
Also, the map induced by $H_m(\mu \circ \alpha)$ gives a chain complex map $\rho\colon M_\bullet^X \to M_\bullet^Y$ such that the following diagram commutes in $\DbLT$:
$$\xymatrix{ {M}_\bullet^{ X} \ar[r]^{\beta^X} \ar[d]^{\rho} & {X}_\bullet \ar[d]^{f\qquad .}\\ {M}_\bullet^{ Y} \ar[r]^{\beta^Y} & {Y}_\bullet }$$ By the property of pullbacks, $H_m(\nu)$ is an epimorphism in $\C$. Thus, $H_m(\beta^X)$ is an epimorphism in $\C$. This completes the proof.
\end{proof}
\begin{lemma}\label{lem:fullness-of-the-map}
    Suppose $\L$ has the SR property with respect to $\T$. Then the functor $\iota \colon D^b(\T \cap \L) \to D^b_{\tiny \L}(\T)$ is full. 
\end{lemma}
    \begin{proof}
        Let $ X_\bullet,  Y_\bullet \in D^b(\T \cap \L)$ and $f \in \Hom_{\tiny \DbLT}({X_\bullet,Y_\bullet})$ and $\min({X_\bullet \oplus Y_\bullet}) = m$. We may also assume that $\min_c( X_\bullet ) \geq m$ and $\min_c( Y_\bullet ) \geq m$. The proof is by induction on $\width(X_\bullet \oplus Y_\bullet) = k$. The base case is when $k = 0$, in which case both $\Hom_{\tiny D^b(\T \cap \L)}({X_\bullet,Y_\bullet})$ and $\Hom_{\tiny \DbLT}({X_\bullet,Y_\bullet})$ are naturally isomorphic to $\Hom_{\C}({X_\bullet,Y_\bullet})$ (the module version is \cite[Lemma 4.2]{SandersSane} and its statement and proof hold true for any abelian category). Now assume the  statement holds for all $k' < k$ and assume that $\width(X_\bullet \oplus Y_\bullet) = k$. Applying \Cref{lem:morphisms-in-derived-category-abelian-category}, using the same notation therein, extending the horizontal chain complex maps $\beta^X$ and $\beta^Y$ to triangles in $D^b(\T \cap \L)$, and using TR3, we get a morphism of triangles in $\DbLT$: 
$$\xymatrix{\Sigma^{-1}{C_\bullet^X} \ar[r]^-{\alpha^X} \ar[d]^{\Sigma^{-1} \lambda} & {M_\bullet^X} \ar[r]^{\beta^X} \ar[d]^{\rho} &  X_\bullet \ar[r]^{\gamma^X} \ar[d]^{f} & {C_\bullet^X} \ar[d]^{\lambda}\\ \Sigma^{-1}{C_\bullet^Y} \ar[r]^-{\alpha^Y} & {M_\bullet^Y} \ar[r]^{\beta^Y} &  Y_\bullet \ar[r]^{\gamma^Y} &{C_\bullet^Y} }$$
where $H_m(\beta^X)$ and $H_m(\beta^Y)$ are surjective. Using the surjectivity, we can see that $\width({C_\bullet^X} \oplus {C_\bullet^Y})<k$ by taking the induced long exact sequence on homology. Therefore, by the induction hypothesis, $\Hom_{\tiny D^b(\T \cap \L)}({C_\bullet^X}, {C_\bullet^Y}) \to \Hom_{\tiny \DbLT}({C_\bullet^X}, {C_\bullet^Y})$ is surjective and hence there exists $\widetilde{\lambda}$ such that $\iota(\widetilde{\lambda}) = \lambda$. Since $\alpha^X , \alpha^Y , \gamma^X , \gamma^Y , \beta^X , \beta^Y$ and $\rho$ are morphisms in $D^b(\T \cap \L)$, TR3 gives a map $g \colon  X_\bullet \to  Y_\bullet$ such that we have the following morphism of triangles in $D^b(\T \cap \L)$: 
$$\xymatrix{ \Sigma^{-1}{C_\bullet^X} \ar[r]^-{\alpha^X} \ar[d]^{\Sigma^{-1} \tilde{\lambda}} & {M_\bullet^X} \ar[r]^{\beta^X} \ar[d]^{\rho} &  X_\bullet \ar@{-->}[d]^{g} \ar[r]^{\gamma^X}  & {C_\bullet^X} \ar[d]^{\tilde{\lambda}}\\ \Sigma^{-1}{C_\bullet^Y} \ar[r]^-{\alpha^Y} & {M_\bullet^Y} \ar[r]^{\beta^Y} &  Y_\bullet \ar[r]^{\gamma^Y} &C_\bullet^Y }$$
By altering $g$ using weak kernel/cokernel properties, we get $f = \iota(h)$ for some $h$, which completes the proof (refer to the end of \cite[Proposition 4.4]{SandersSane}).
\end{proof}

 Now we state the main result of this section.
\begin{thm}\label{thm:derived-equivalence-main-abelian-category}
 Suppose $\L \subseteq \mathscr C$ has the SR property with respect to a thick subcategory $\T$. Then the functor $\iota \colon D^b(\T \cap \L) \to D^b_{\tiny \L}(\T)$ is an equivalence.    
\end{thm}
\begin{proof}
    Clearly, the functor $\iota$ is faithful on objects. Thus, from \cite[Lemma 2.19]{SandersSane} and \Cref{lem:fullness-of-the-map}, it suffices to show that $\iota$ is essentially surjective. Let $ P_\bullet \in D^b_{\tiny \L}(\T)$. We will show by induction on $\width( P_\bullet) = k$ that there exists $\widetilde{ P_\bullet} \in D^b(\T \cap \L)$ such that $\iota(\widetilde{ P_\bullet}) \cong  P_\bullet$. The base case is $k=0$. If $P_\bullet$ is acyclic, there is nothing to prove. If $H_m(P_\bullet) \neq 0$, we choose $\widetilde{ P_\bullet} = \Sigma^m H_m(P_\bullet)$ and we are done.  
    \par Let us assume the induction statement holds for all $k'<k$. Let $ P_\bullet \in D^b_{\tiny \L}(\T)$ with $\width( P_\bullet) = k$. Set $m = \min( P_\bullet)$. Since $\L$ has the SR property with respect to $\T$, there exist $T_\bullet \in Ch^b_{\tiny \L}(\T)$ and a chain complex map $\alpha \colon T_\bullet \to  P_\bullet$ such that $\min_c(T_\bullet) = m$, $\supph(T_\bullet) = \{m\}$ and $H_m(\alpha) \colon H_m(T_\bullet) \to  H_m(P_\bullet)$ is an epimorphism in $\C$. We extend $\alpha$ to an exact triangle $$\Sigma^{-1}  C_\bullet \xrightarrow[]{\gamma}  T_\bullet \xrightarrow[]{\alpha}  P_\bullet \xrightarrow{\delta} C_\bullet.$$
    From the long exact sequence induced by the triangle, it follows that $\width( C_\bullet)<k$. Thus the induction hypothesis gives that $\iota(\widetilde{ C_\bullet}) \cong {C_\bullet}$ for some $\widetilde{ C_\bullet} \in D^b(\T \cap \L)$. Note that $H_m(T_\bullet) \in \T \cap \L$. Since $\iota$ is full by \Cref{lem:fullness-of-the-map}, the map $$\Hom_{\tiny D^b(\T \cap \L)}(\Sigma^{-1} \widetilde{ C_\bullet}, \Sigma^{m}H_m(T_\bullet)) \to \Hom_{\tiny \DbLT}(\Sigma^{-1} \widetilde{ C_\bullet}, \Sigma^{m}H_m(T_\bullet))$$ is surjective. Hence there exists $\widetilde{\gamma} \colon \Sigma^{-1} \widetilde{ C_\bullet} \to \Sigma^{m}H_m(T_\bullet)$ such that $\iota(\widetilde{\gamma}) = \gamma$. Set $\widetilde{ P_\bullet}$ to be the cone of $\widetilde{\gamma}$ in $D^b(\T \cap \L)$. Then it follows from TR3 that $\iota(\widetilde{ P_\bullet}) \cong  P_\bullet$. This completes the proof.
\end{proof}
    \begin{cor}\label{cor:fully-faithful-thick}
        Suppose $\L$ has the SR property with respect to $\T$. Further assume that $\C$ contains enough projectives and $\T$ contains all the projective objects in $\C$. Then all the functors $ D^b(\T \cap \L) \to D^b_{\tiny \L}(\T) \to D^b(\T) \to D^b(\C) \to D(\C)$ are fully faithful. 
    \end{cor}
\begin{proof}
    Applying \Cref{thm:derived-equivalence-main-abelian-category} yields that the first functor is fully faithful. The second functor is fully faithful by definition. The third and fourth functors being fully faithful is a very general property, and follows from \Cref{rmk:fully-faithful-thicksub-in-abelian}(i) and \cite[Lemma 4.1.16]{HenningKrause}. 
\end{proof}

\section{Resolving subcategories and their properties}\label{sec:resolving-subcategories}
We define resolving subcategories as in \cite[Definition 5.1]{Ruben-Sondre-Roosmalen}. Note that this definition is more general than the classical definition in Auslander-Bridger (\cite[page 99]{Auslander-Bridger}).
\begin{defn}\label{defn:resolving-subcategory-abelian-category}\label{defn:resolving-subcategory}
    A full subcategory $\CA$ of an abelian category $\C$ is defined to be a \emph{resolving subcategory} of $\C$ if it satisfies the following conditions :
\begin{enumerate}
    \item[(R1)] There are enough $\CA$-objects, i.e., given an object $C \in \mathscr{C}$, there exists an object $A \in \CA$ and an epimorphism $A \to C$ in $\C$.
    \item[(R2)]  For any $A_1 , A_2 \in \mathscr{C}$, $A_1 \oplus A_2 \in \CA$ if and only if $A_1$ and $A_2$ lie in $\CA$.
    \item[(R3)]  For any short exact sequence $0 \to A_1 \to A_2 \to A_3 \to 0$ with $A_3 \in \CA$, then $A_1 \in \CA$ if and only if $A_2 \in \CA$.
\end{enumerate}
\end{defn}
Note that this definition does not assume that resolving subcategories contain projective objects. As a result, many statements and proofs have the same flavor as classical ones with projective objects, but there are subtle differences. Hence, we sketch many of the proofs and give details only where the differences occur. We need this generality because, in \Cref{sec: der equiv quasi proj} we work with $Coh(X)$ for a quasi-projective variety $X$, and this category does not have enough projectives, but does have enough locally free sheaves. The remark below gathers together all the relevant constructions that follow from this definition.
\begin{rmk}\label{rmk:resolving-subcategories}
\begin{enumerate}
    \item $\CA$-resolutions of any object exist.
    \item Given a morphism $C \xrightarrow{f} C''$ and an $\CA$-resolution $A_{\bullet}''$ of $C''$, there exists an $\CA$- resolution $A_{\bullet}$ of $C$ and a chain complex morphism $A_{\bullet} \xrightarrow{f'} A_{\bullet}''$ of $C''$ making the diagram commute.
    \begin{proof}
        Consider the pullback of $f$ and $A_0'' \rightarrow C''$, and then map an object of $\CA$ on it. This defines $A_0$ and the augmentation map. Consider the kernel of $A_0 \rightarrow C$ and note that it has a natural map to $Z_0(A_{\bullet}'')$. We then continue the process to inductively obtain $A_{\bullet}$.
    \end{proof}
    \item If the morphism $f$ in (2) is an epimorphism, then the constructed morphism $f'$ is also an epimorphism in $Ch^{+}(\C)$.
    \begin{proof}
        This follows from the properties of the pullback.
    \end{proof}
    \item Given a short exact sequence $0 \to C' \xrightarrow{g} C \xrightarrow{f} C'' \to 0$, and an epimorphism of $\CA$-resolutions $A_{\bullet} \xrightarrow{f'} A_{\bullet}''$ in $Ch^{+}(\C)$ lifting $f$, $\ker(f')$ is an $\CA$-resolution of $C'$ and $\ker(f') \to A_{\bullet} $ lifts $g$.
    \begin{proof}
        The kernel of $f'$ is clearly a chain complex of objects in $\CA$, and the construction of the augmentation map $\ker(f') \to C'$ shows that $\ker(f') \to A_{\bullet} $ lifts $g$. Finally, from the long exact sequence of homology, it follows that $\ker(f')$ is an $\CA$-resolution of $C'$.
    \end{proof}
    \item Given a short exact sequence $0 \to C' \xrightarrow{g} C \xrightarrow{f} C'' \to 0$, there is a short exact sequence of their $\CA$-resolutions 
    $0 \to A_{\bullet}' \xrightarrow{g'} A_{\bullet} \xrightarrow{f'} A_{\bullet}'' \to 0$ where $f' , g'$ are lifts of $f,g$ respectively.
    \begin{proof}
        By (1), there is an $\CA$-resolution $A_{\bullet}''$ of $C''$. Applying (3) yields an $\CA$- resolution $A_{\bullet}$ of $C$ and an epimorphism $A_{\bullet} \xrightarrow{f'} A_{\bullet}''$ in $Ch^{+}(\C)$ which lifts $f$. From (4), $\ker(f')$ is an $\CA$-resolution of $C'$. Defining this kernel to be $A_{\bullet}'$ and its inclusion in $A_{\bullet}$ to be $g'$ completes the proof.
    \end{proof}
    \item Suppose there are morphisms $C' \xrightarrow{g} C \xrightarrow{f} C''$, such that the composition $f \circ g = 0$, and a morphism of $\CA$-resolutions $A_{\bullet} \xrightarrow{f'} A_{\bullet}''$ which lifts $f$, and such that $f'$ surjectively factors through an $\CA$-resolution of $\im(f)$. Then there is an $\CA$-resolution $A_{\bullet}'$ of $C'$ and a morphism $A_{\bullet}' \xrightarrow{g'} A_{\bullet}$ which lifts $g$ and surjectively factors through an $\CA$-resolution of $\im(g)$. In particular, $f' \circ g' = 0$.
    \begin{proof}
        We consider the short exact sequence $0 \to \ker(f) \to C \to \im(f) \to 0$ and let $D_{\bullet}$ be the $\CA$-resolution of $\im(f)$ through which $f'$ surjectively factors. Since $A_{\bullet} \to D_{\bullet}$ is an epimorphism in $Ch^{+}(\C)$, invoking (4), we see that the kernel $D_{\bullet}'$ is an $\CA$-resolution of $\ker(f)$. Since $f \circ g = 0$, it follows that $\im(g) \subseteq \ker(f)$. Hence, by (2), we can construct 
        an $\CA$-resolution $D_{\bullet}''$ of $\im(g)$ and a morphism from $D_{\bullet}'' \rightarrow D_{\bullet}'$ which lifts the inclusion. Again, by (3), we obtain an $\CA$-resolution $A_{\bullet}'$ of $C'$ and an epimorphism $A_{\bullet}' \to D_{\bullet}''$ in $Ch^{+}(\C)$ lifting the morphism $C' \to \im(g)$. Now, letting $g'$ be the composition $A_{\bullet}' \to D_{\bullet}'' \to D_{\bullet}' \to A_{\bullet}$, we see that it is a lift of $g$ and that it surjectively factors through $D_{\bullet}''$ which is an $\CA$-resolution of $\im(g)$ as required. This completes the proof.
    \end{proof}
\end{enumerate}
\end{rmk}
The next lemma is the usual one used to define the dimension with respect to the resolving category $\CA$, following on the lines of \cite[Lemma 3.2]{Auslander-Bridger}. However, there is a slight subtlety in the proof since we no longer have projective resolutions.
\begin{lemma}\label{A-dim}
    The following statements are equivalent for an object $C \in \mathscr{C}$:
    \begin{enumerate}
        \item There exists an $\CA$-resolution of $C$ of length $n$.
        \item For every $\CA$-resolution $A_{\bullet}$ of $C$, $Z_n(A_{\bullet}) \in \CA$.
    \end{enumerate}
\end{lemma}
\begin{proof}
    We give a sketch of the proof. (2) implies (1) is clear. So we prove (2) assuming (1). Let $A'_{\bullet}$ be an $\CA$-resolution of $C$ of length $n$, and $A_{\bullet}$ be any $\CA$-resolution of $C$. We consider the canonical/soft truncation $D_{\bullet} = \tau_{\leq n}A_{\bullet}$. Note that $D_n = Z_n(A_{\bullet})$ which we want to prove is in $\CA$. \\
    Case 1 : Suppose there is a chain complex morphism $f \colon A'_{\bullet} \rightarrow D_{\bullet}$ which lifts the identity map on $C$. Then the mapping cone of $f$ is exact.
    Since the terms in degrees $0$ to $n-1$ are in $\CA$, it follows that $\ker(D_{n-1} \oplus A'_{n-2} \to D_{n-2} \oplus A'_{n-3})$ is in $\CA$. Also, there is a short exact sequence
    $$0 \to A'_n \to D_n \oplus A'_{n-1} \to \ker(D_{n-1} \oplus A'_{n-2} \to D_{n-2} \oplus A'_{n-3}) \to 0$$ and since all but $D_n$ are in $\CA$, so is $D_n$. \\
    Case 2 : Suppose there is a chain complex morphism $g \colon D_{\bullet} \rightarrow A'_{\bullet}$ which lifts the identity map on $C$. Then the mapping cone of $g$ is exact. 
    Since the terms in degrees $0$ to $n-1$ are in $\CA$, it follows that $\ker(A'_{n-1} \oplus D_{n-2} \to A'_{n-2} \oplus D_{n-3})$ is in $\CA$. 
    Also, there is a short exact sequence
    $$0 \to D_n \to A'_n \oplus D_{n-1} \to \ker(A'_{n-1} \oplus D_{n-2} \to A'_{n-2} \oplus D_{n-3}) \to 0$$ and since all but $D_n$ are in $\CA$, so is $D_n$. \\
    Case 3 : For the general case, consider the $\CA$-resolution $A'_{\bullet} \oplus A_{\bullet} $ of $C \oplus C$ and the diagonal map $C \to C \oplus C$. Using \Cref{rmk:resolving-subcategories}(2), we obtain an $\CA$-resolution $A''_{\bullet}$ of $C$ such that it lifts the diagonal map. Taking projections, we obtain chain complex morphisms $f \colon A''_{\bullet} \rightarrow A_{\bullet}$ and $g \colon A''_{\bullet} \rightarrow A'_{\bullet}$ which lift the identity maps on $C$. Then invoking case 1 for $f$, we get that $Z_n(A''_{\bullet})$ is in $\CA$ and invoking case 2 for $g$, we get that $Z_n(A'_{\bullet})$ is in $\CA$, as required. This completes the proof.
\end{proof}

As usual, we now define the least number $n$ such that there is an $\CA$-resolution of $C$ of length $n$ as $\dim_{\tiny \CA}(C)$ (defined to be $+\infty$ if a finite $\CA$-resolution does not exist). As a standard consequence, we obtain the following corollary.
\begin{cor}\label{cor:inequalities-of-resolving-dim}
    Let $0 \to C' \xrightarrow{g} C \xrightarrow{f} C'' \to 0$ be a short exact sequence in $\mathscr C$. Then the following inequalities hold :
    \begin{enumerate}
        \item $\dim_{\tiny \CA}(C') \leq \max\{\dim_{\tiny \CA}(C), \dim_{\tiny \CA}(C'')-1\}$.
        \item $\dim_{\tiny \CA}(C) \leq \max\{\dim_{\tiny \CA}(C'), \dim_{\tiny \CA}(C'')\}$.
        \item $\dim_{\tiny \CA}(C'') \leq \max\{\dim_{\tiny \CA}(C), \dim_{\tiny \CA}(C')+1\}$.
    \end{enumerate}
\end{cor}
\begin{proof}
    This follows from the existence of a short exact sequence of resolutions as in \Cref{rmk:resolving-subcategories}(5) and using the long exact sequence of homology thus obtained.
\end{proof}

\section{Derived equivalences for closures of resolving subcategories}\label{sec:resolving-der-eq}
We begin with the definition of the closure of a resolving subcategory in an abelian category. Several interesting thick subcategories arise as the closure of a resolving subcategory.
\begin{defn}
    Let $\overline{\CA}$ denote the full subcategory of $\C$ with objects $$\{ C \in \mathscr C \mid \dim_{\tiny \CA}(C) < \infty \}.$$ This is the smallest thick subcategory containing $\CA$ and we call it the \emph{closure} of $\CA$.
\end{defn}
The statement and proof of the next theorem are well known when the resolving subcategories contain enough projectives. The proof uses the same arguments as the ones made with the extra hypothesis of enough projectives, along with \Cref{rmk:resolving-subcategories} and some minor checking. Thus, we sketch the proof, giving details only for the points which need to be checked.
\begin{thm}\label{thm:equivalence-of-Abar-and-A}
    Let $\CA$ be a resolving subcategory and $\L$ a Serre subcategory of an abelian category $\mathscr C$. Then the natural inclusion functor $Ch^b_{\tiny \L}(\CA) \to Ch^b_{\tiny \L}(\overline{\CA}) $ induces a derived equivalence $D^b_{\tiny \L}(\CA) \simeq D^b_{\tiny \L}(\overline{\CA}) $.
\end{thm}
\begin{proof}
     Given $A_{\bullet} \in Ch^b_{\tiny \L}(\overline{\CA})$, we first construct a bounded double complex $C_{\bullet\bullet}$ with objects in $\CA$ such that the $n^{th}$ column resolves $A_n$ by repeatedly applying \Cref{rmk:resolving-subcategories}(6). Let $m = \min_c(A_{\bullet})$ and $m'=\max_c(A_{\bullet})$. For $i < m$ and $i > m'$, choose $C_{ i \bullet }'$ to be the zero chain complex. By \Cref{rmk:resolving-subcategories}(1), we obtain a resolution $C_{m\bullet}'$ of $A_m$. Clearly the hypothesis of \Cref{rmk:resolving-subcategories}(6) is satisfied for the part of the complex $A_{m+1} \to A_m \to A_{m-1}$ with $\CA$-resolutions $C_{m\bullet}' \to C_{{m-1}{\bullet}}'$, and thus, we inductively obtain the chain complexes $C_{n\bullet}'$ and morphisms between them for $m \leq n \leq m'$. Thus, we get a double complex $C_{\bullet\bullet}'$ with objects in $\CA$ such that the $n^{th}$ column resolves $A_n$ and $C_{ij}' = 0$ for $i < m$ and $i > m'$. We now use that $A_n \in \overline{\CA}$ and let $C_{t\bullet}$ be the soft truncation of $C_{t\bullet}'$ at degree $\max \{ \dim_{\tiny \CA} (A_n) \mid m \leq n \leq m' \}$ for each $t$. By \Cref{A-dim}, it is clear that $C_{\bullet\bullet}$ is a double complex in $\CA$ and satisfies the other required conditions. We now consider the total complex $D_{\bullet}$ of $C_{\bullet\bullet}$ and the usual proof using spectral sequences, shows that the natural morphism of chain complexes $D_{\bullet} \to A_{\bullet}$ is a quasi-isomorphism. Thus, we see that for every chain complex $A_{\bullet} \in Ch^b_{\tiny \L}(\overline{\CA})$, there exists a 
    quasi-isomorphism $D_{\bullet} \to A_{\bullet}$ where $D_{\bullet} \in Ch^b_{\tiny \L}(\CA)$. Hence $D^b_{\tiny \L}(\CA) \to D^b_{\tiny \L}(\overline{\CA}) $ is essentially surjective. Since the construction above shows that any morphism in $D^b_{\tiny \L}(\overline{\CA})$ is equivalent to one in which the middle object is in $Ch^b_{\tiny \L}(\CA)$, it follows that $D^b_{\tiny \L}(\CA) \to D^b_{\tiny \L}(\overline{\CA}) $ is also 
    fully faithful.
\end{proof}
Since $\overline{\CA}$ is a thick subcategory, we apply \Cref{thm:derived-equivalence-main-abelian-category} to get the following result.
\begin{thm}\label{thm:derived-equivalence-AbarcapL-AbarsuppL-AsuppL}
    Suppose a Serre subcategory $\L$ has the SR property with respect to $\overline{\CA}$ for a resolving subcategory $\CA$. Then $D^b(\overline{\CA} \cap \L) \simeq D^b_{\L}(\overline{\CA}) \simeq D^b_{\L}(\CA)$.
\end{thm}
For $g \in \N$, let $\overline{\CA}^{\leq g}$ be the full subcategory of $\C$ consisting of objects $\{C \in \mathscr C \mid \dim_{\tiny \CA}(C)\leq g\}$. Then $\overline{\CA}^{\leq g}$ is an exact subcategory of $\mathscr C$ and $\overline{\CA}^{\leq g} \cap \L \to \overline{\CA} \cap \L$ is the natural inclusion of exact categories. The following lemma provides a criterion for determining when the induced map on their bounded derived categories is an equivalence.

\begin{lemma}\cite[Section 1.5]{Keller}\label{lem:Keller-lemma}
    Let $\mathscr{E}_1$ and $\mathscr{E}_2$ be full exact subcategories of $\C$ such that $\mathscr{E}_1 \subseteq \mathscr{E}_2$. Consider the conditions 
    \begin{enumerate}[(a)]
        \item For each $B \in \mathscr E_2$, there is an acyclic chain complex of $\mathscr{E}_2$
        $$0 \to B \to A_0 \to A_1 \to \cdots \to A_n \to 0$$
        where each $A_i \in \mathscr E_1$.
        \item For each short exact sequence $0 \to B' \to B \to B'' \to 0$ in $\mathscr E_2$ with $B' \in \mathscr E_1$, there is a commutative diagram
        $$\xymatrix{ 0 \ar[r] & B' \ar[r]^{} \ar@<-.2ex>@{-}[d]\ar@<.2ex>@{-}[d] & B \ar[d]^{} \ar[r] & B'' \ar[d] \ar[r] & 0 \\ 0 \ar[r] & B' \ar[r] & A \ar[r] & A'' \ar[r] & 0}$$
        where the second row is exact in $\mathscr{E}_1$. 
    \end{enumerate}
    If conditions $(a)$ (or its dual) and $(b)$ (or its dual) hold, then $D^b(\mathscr{E}_1) \simeq D^b(\mathscr{E}_2)$.
\end{lemma}
\begin{prop}\label{thm:derived-eq-keller-abelian-category}
    Let $g \in \N$. Suppose a resolving subcategory $\CA$ of $\mathscr C$ has the property that for each $C \in \overline{\CA} \cap \L$, there exist $C' \in \overline{\CA}^{\leq g} \cap \L$ and an epimorphism $C' \to C$ in $\C$. Then 
    $$D^b(\overline{\CA}^{\leq g} \cap \L) \simeq D^b(\overline{\CA} \cap \L).$$
\end{prop}
\begin{proof}
    We show that the dual statements of $(a)$ and $(b)$ in \Cref{lem:Keller-lemma} hold for the inclusion functor $\overline{\CA}^{\leq g} \cap \L \subseteq \overline{\CA} \cap \L$. We first show that the dual of $(a)$ holds. Let $C \in \overline{\CA} \cap \L$. If $\dim_{\tiny \CA}(C) \leq g$, then there is nothing to prove. Suppose $\dim_{\tiny \CA}(C) > g$. Then by assumption, there exist $C_0 \in \overline{\CA}^{\leq g} \cap \L$ and an epimorphism $C_0 \to C$ in $\C$. Let $K_0 = \ker(C_0 \to C)$. Since $\L$ is a Serre subcategory, $K_0 \in \L$. By using \Cref{cor:inequalities-of-resolving-dim}, $\dim_{\tiny \CA}(K_0)<\dim_{\tiny \CA}(C)$. Doing this process $n = (\dim_{\tiny \CA}(C) - g)$-times, $K_{n-1} \in \overline{\CA}^{\leq g} \cap \L$, we obtain the required acyclic chain complex $$0 \to K_{n-1} \to C_{n-1} \to \cdots \to C_1 \to  C_0 \to C \to 0.$$
    To show the dual of (b), let $0 \to C' \xrightarrow{f} C \xrightarrow{g} C'' \to 0$ be a short exact sequence in $\overline{\CA} \cap \L$ with $C'' \in \overline{\CA}^{\leq g} \cap \L$. By hypothesis, there exists $D \in \overline{\CA}^{\leq g} \cap \L$ such that $D \xrightarrow{h} C$ is an epimorphism in $\C$ and let $D' = \ker(g \circ h)$. Since $\overline{\CA}^{\leq g} \cap \L$ is closed under kernels of epimorphisms, $D' \in \overline{\CA}^{\leq g} \cap \L$. Thus, we have the following commutative diagram 
$$\xymatrix{ 0 \ar[r] & D' \ar[r]^{} \ar[d] & D \ar[d]^{h} \ar[r]^{g \circ h} & C'' \ar@<-.2ex>@{-}[d]\ar@<.2ex>@{-}[d] \ar[r] & 0 \\ 0 \ar[r] & C' \ar[r]^{f} & C \ar[r]^{g} & C'' \ar[r] & 0}$$
    where the first row is a short exact sequence in $\overline{\CA}^{\leq g} \cap \L$. Hence, the proof is complete.
\end{proof}

Combining \Cref{thm:derived-equivalence-AbarcapL-AbarsuppL-AsuppL} and \Cref{thm:derived-eq-keller-abelian-category} yields the following result:
\begin{thm}\label{thm:derived-equivalence-for-abelian-category-resolving-sub}
    Let $\L$ be a Serre subcategory and $\CA$ a resolving subcategory of $\C$. 
    Suppose $\L$ has the SR property with respect to $\overline{\CA}$, and for each $C \in \overline{\CA} \cap \L$, there exist $C' \in \overline{\CA}^{\leq g} \cap \L$ and an epimorphism $C' \to C$ in $\C$ for some $g \in \N$. Then  $$D^b(\overline{\CA}^{\leq g} \cap \L) \simeq D^b(\overline{\CA} \cap \L) \simeq D^b_{\L}({\overline{\CA}}) \simeq D^b_{\L}({\CA}).$$
\end{thm}
\begin{rmk}
    Let $\L$ and $\CA$ be as in \Cref{thm:derived-equivalence-for-abelian-category-resolving-sub}. Further assume that $\C$ has enough projectives and $\CA$ contains all projective objects of $\C$. Then it follows from \Cref{thm:derived-equivalence-for-abelian-category-resolving-sub} and \Cref{cor:fully-faithful-thick}.
    that all the functors below are fully faithful : $$D^b(\overline{\CA}^{\leq g} \cap \L) \xrightarrow{\sim} D^b(\overline{\CA} \cap \L)\xrightarrow{\sim} D^b_{\L}(\overline{\CA}) \to D^b(\overline{\CA}) \to D^b({\C}) \to D(\C).$$
\end{rmk}

\section{The derived equivalences for $\starmods$}\label{sec: der equiv graded}
    Throughout this section, let $S = \Oplus_{i \in \mathbb Z}S_i$ be a graded ring, which is also noetherian, that is, every ascending chain of ideals stabilizes. Note that the noetherian hypothesis is necessary since it is crucially used in \cite{Foxby-Halvorsen}, \cite{SandersSane}, and \cite{GanapathySane}, all of which we require in our proofs. The main goal of this section is to show that there is a large collection of Serre subcategories and thick subcategories of $\starmods$ satisfying the hypotheses of the theorems in Sections \ref{sec:derived-equivalence-for-abelian-category} and \ref{sec:resolving-der-eq}, and hence yielding the derived equivalences proved in those sections. 
    \begin{defn}\label{defn:equivalent-filtrations-ideals}
    Let $\{ I_t \}_{t \in \tiny \N}$ and $\{ {J}_t \}_{t \in \tiny \N}$ be filtrations of ideals in $S$. We say the filtrations $\{ I_t \}_{t \in \tiny \N}$ and $\{ {J}_t \}_{t \in \tiny \N}$ are \emph{equivalent} if for each $t \in \N$ there exists $s$ such that ${J}_s \subseteq I_t$ and for each $s \in \N$, there exists $t$ such that $I_t \subseteq {J}_s$.
\end{defn}
 For the rest of the section, let $\CA$ denote the resolving subcategory of $\starmods$ and $\mathscr T$ be a thick subcategory containing $\CA$.
\begin{defn}\label{defn:efpd-graded-setup}
    Let $V \subseteq \proj(S)$ be a closed set. We say that $V$ has \emph{eventually finite $\CA$-dimension} if there exists a filtration of homogeneous ideals $\{I_t\}$ which is equivalent to $\{I_1^s\}$ such that for each $t \in \N$, $\dim_{\tiny \CA}(I_t)<\infty$ and $\Supp_S(S/I_t) = V$.
\end{defn}
\begin{rmk}
   An ideal in a commutative noetherian ring having eventually finite projective dimension (efpd) has been defined in \cite[Definition 2.14]{GanapathySane}. Since efpd is a property up to radical, in this article we define eventually finite $\CA$-dimension for a closed set instead of an ideal.
\end{rmk}
We now consider resolving subcategories $\CA$ of $\starmods$ which contain the object $S$ and are closed under shifts. Note that any such $\CA$ will contain $\starf(S)$.
\begin{thm}\label{thm:graded-efpd-implies-SR-prop}
    Let $I \subseteq S$ be a homogeneous ideal, $\L = \L_{V(I)}$, and $\T$ a thick subcategory of $\starmods$. Suppose $\T$ contains a resolving subcategory $\CA$ of $\starmods$ which contains $S$ and is closed under shifts, and $V(I)$ has eventually finite $\CA$-dimension. Then $\L$ has the SR property with respect to $\T$.
\end{thm}
\begin{proof}
    Let $X_\bullet \in Ch^b_{\tiny \L}(\starmods)$ which is not exact with $\min(X_\bullet) = m$ and a morphism $f \colon X_m \to Q$ with $Q \in \L$. Without loss of generality, we may assume $\min(X_\bullet) = 0$. By \Cref{lem:Cartan-Eilenberg}, there exists a (graded) quasi-isomorphism $\gamma \colon P_{\bullet} \to X_{\bullet}$ where $P_{\bullet} \in Ch^{+}_{\tiny \L}(\starf(S))$. Applying \Cref{lem:Tate-to-bdd-above-complex} to $P_\bullet$ and $\delta = (f \circ \gamma_0) \colon P_0 \to Q$, we get a chain complex map $\psi \colon U_\bullet \to P_\bullet$ in $Ch^{+}_{\tiny \L}(\starf(S))$ such that the properties $(i)-(v)$ in \Cref{lem:Tate-to-bdd-above-complex} holds. In particular, $(i)-(iii)$ show that $U_\bullet$ is a $\starf(S)$-resolution of $S/J \otimes_S P_0$ for some ideal $J$ with $\sqrt{I} = \sqrt{J}$. Since $V(I)$ has eventually finite $\CA$-dimension, there exists a homogeneous ideal $I'\subseteq J$ such that $\sqrt{I'} = \sqrt{J}$ and $\dim_{\tiny \CA}(I')<\infty$. Therefore, there exists $T_{\bullet} \in Ch^{b}_{\tiny \L}(\T)$, an $\CA$-resolution of $S/I' \otimes P_0$, along with a graded chain complex map $\beta \colon T_\bullet \to U_\bullet$ which lifts the natural surjection $S/I' \otimes P_0 \to S/J \otimes P_0$. Clearly, $(T_\bullet, \gamma \circ \psi \circ \beta)$ is a strong reducer of $(X_\bullet,f)$ and this completes the proof. 
\end{proof}
Recall that $W \subseteq \proj(S)$ is called a \emph{specialization closed} set if it is an arbitrary union of closed sets of $\proj(S)$. For a specialization closed set $W$, we denote $\L_W = \{ M \in \starmods \mid \Supp_S(M) \subseteq W\}$, which is a Serre subcategory of $\starmods$. Then the above theorem can be strengthened as follows: 
\begin{cor}\label{rmk:SRproperty-extension-graded} Let $\T$ and $\CA$ be as in \Cref{thm:graded-efpd-implies-SR-prop}. Suppose $W$ is a specialization closed set of $\proj(S)$ such that for each closed set $W'\subseteq W$, there exists a closed set $W'' \subseteq \proj(S)$ with $W' \subseteq W'' \subseteq W$ and $W''$ has eventually finite $\CA$-dimension. Then $\L_W$ has the SR property with respect to $\T$. 
\end{cor}
\begin{proof}
      Let $X_\bullet \in Ch^b_{\L_W}(\starmods)$ which is not exact with $\min(X_\bullet) = m$ and a morphism $f \colon X_m \to Q$ with $Q \in \L$. 
     Consider the closed set $W' = \cup_i \Supp_S(H_i(X_\bullet)) \cup \Supp_S(Q) \subseteq W$. By hypothesis, there exists a closed set $W''$ such that $W' \subseteq W'' \subseteq W$ and $W''$ has eventually finite $\CA$-dimension. By \Cref{thm:graded-efpd-implies-SR-prop}, $(X_\bullet,f)$ has a strong reducer with respect to $\L_{W''}$ and $\T$. Since $\L_{W''} \subseteq \L_{W}$, it follows that $(X_\bullet,f)$ has a strong reducer with respect to $\L_{W}$ and $\T$. 
\end{proof}
\begin{thm}\label{cor:derived-equiv-charp-finite-projdim}
    Let $V(I) \subseteq \proj(S)$ be a closed set for a homogeneous ideal $I \subseteq S$ and $\L = \L_{V(I)}$. Suppose $\ch(S) = p>0$ for a prime $p$ and $\pd_S(S/I) < \infty$. Then, for any resolving subcategory $\CA$ of $\starmods$ which contains $S$ and is closed under shifts, there are derived equivalences
    \begin{align*}
    D^b(\overline{\CA} \cap \L) \simeq D^b_{\tiny \L}(\overline{\CA}) \simeq D^b_{\tiny \L}(\CA).
\end{align*} 
\end{thm}
\begin{proof}
    Since $\ch(S) = p>0$, it follows from \cite[Theorem 1.7]{Peskine_Szpiro_thesis} that $\pd_S(S/I) = \pd_S(S/I^{[p^e]})$ for all $e \in \N$, where $I^{[p^e]}$ denotes the $e$-th Frobenius power of $I$. Since $\{I^{[p^e]}\}_{e \in \tiny \N}$ is equivalent to $\{I^n\}_{n \in \tiny \N}$, $V(I)$ has eventually finite $\starp(S)$-dimension and hence eventually finite $\CA$-dimension. Therefore, $\L$ has the SR property with respect to $\overline{\CA}$ and thus, from \Cref{thm:derived-equivalence-AbarcapL-AbarsuppL-AsuppL}, we get the required equivalences.
\end{proof} 

\begin{defn}
    For a closed set $V = V(I)$ for an ideal $I \subseteq S$, let $$S(V) = \{ \{I_t\}_{t \in \tiny \N} \mid \{I_t\}_{t \in \tiny \N} \text{\ is\ equivalent\ to\ } \{I^n\}_{n \in \tiny \N}\}.$$ Note that $S(V)$ only depends on $V$ and does not depend on the choice of the ideal $I$. Now we define $$g(V,\CA) = \inf\limits_{\{I_t\} \in S(V)} \limsup\limits_{t \in \tiny \N} \{ \dim_{\tiny \CA}(S/I_t)  \}.$$
\end{defn}
\begin{rmk}\label{rmk:grade-leq-g}
    Recall that for any ideal $I \subseteq S$, $\grade(I) \leq \pd_S(S/I)$ and $\grade(I) = \grade(\sqrt{I})$. Therefore, $\grade(I) \leq \pd_S(S/I_t)$ for all ideals $I_t$ whose radical is $\sqrt{I}$. Hence, for $\grade(I) \leq g(V,\starp(S))$. 
\end{rmk}
\begin{thm}\label{thm:}
    Let $V \subseteq \proj(S)$ be a closed set, $\L = \L_{V}$, and $\CA$ a resolving subcategory of $\starmods$ which contains $S$ and is closed under shifts. Set $g = g(V,\CA)$. Then $$D^b(\overline{\CA}^{\leq g} \cap \L) \simeq D^b(\overline{\CA} \cap \L).$$
\end{thm}
\begin{proof}
    If $g$ is infinite, there is nothing to prove. So we may assume $g<\infty$. Let $M \in \overline{\CA} \cap \L$. Now we show that there exist $N \in \overline{\CA}^{\leq g} \cap \L$ and a surjection $N \to M$ in $\starmods$. Then \Cref{thm:derived-eq-keller-abelian-category} yields the required equivalence of derived categories. Since $M \in \L$, we know $\Supp_S(M) \subseteq V$ and hence there exists an ideal $I \subseteq S$ such that $\Supp_S(S/I) = V$ and $I \subseteq \Ann_S(M)$. Thus we have a surjection $\Oplus_{i = 1}^n S/I(-m_i) \to M$ for some $m_i,n \in \N$. By definition of $g$, there exists an ideal $I_t \subseteq I$ such that $\dim_{\tiny \CA}(S/I_t) = g$. Therefore, we get a surjections $\Oplus_{i = 1}^n S/I_t(-m_i)  \to \Oplus_{i = 1}^n S/I(-m_i) \to M$. Taking $N = \Oplus_{i = 1}^n S/I_t(-m_i)$ proves the claim.
\end{proof}
Recall that a homogeneous ideal $I \subseteq S$ is said to be \emph{perfect} if $\grade(I) = \pd_S(S/I)$. The following theorem is about the full subcategory of perfect modules supported on $V(I)$ for a perfect ideal $I$, which is exactly the full subcategory $\starp(S)^{\leq \grade(I)} \cap \L_{V(I)}$ of $\starmods$.

\begin{thm}\label{thm:graded-der-eq-starp}
    Let $V(I) \subseteq \proj(S)$ be a closed set for a homogeneous ideal $I \subseteq S$ and $\L = \L_{V(I)}$. Suppose $\ch(S) = p>0$ for a prime $p$ and $I$ is a perfect ideal. Then, $$D^b(\overline{\starp(S)}^{\leq \grade(I)} \cap \L) \simeq D^b(\overline{\starp(S)} \cap \L) \simeq D^b_{\tiny \L}(\overline{\starp(S)}) \simeq D^b_{\tiny \L}(\starp(S)).$$
\end{thm}
\begin{proof}
 The second and third equivalences follow from \Cref{cor:derived-equiv-charp-finite-projdim}, since perfect ideals have finite projective dimension. Also, $\pd_S(S/I) = \pd_S(S/I^{[p^e]})$ for all $e \in \N$ implies that $g(V,\starp(S)) \leq \pd_S(S/I)$. Hence, we have the inequalities $$\grade(I) \leq g(V,\starp(S)) \leq \pd_S(S/I) = \grade(I)$$ where the first inequality follows from \Cref{rmk:grade-leq-g} and the last equality holds as $I$ is perfect. Therefore, $g(V,\starp(S)) = \grade(I)$ and we apply \Cref{thm:} to get the required result.
\end{proof}

\section{The derived equivalences for quasi-projective schemes}\label{sec: der equiv quasi proj}
In this section, $X$ is a quasi-projective scheme over a noetherian affine scheme and $\CA$ is a resolving subcategory of $Coh(X)$. For a closed subscheme $V \subseteq X$, the corresponding ideal sheaf is denoted by $\I_V$. Throughout, we fix an embedding of $X$ in a projective scheme $Y$. Set $\OX(n) \coloneqq \mathcal {O}_Y(n)|_{X}$. We say $\CA$ is closed under shifts, if for each $\mathcal F\in \CA$ and $n \in \Z$, $\mathcal F \otimes_{X} \OX(n) \in \CA$. Like the previous section, the main goal of this section is to show that the derived equivalences proved in Sections \ref{sec:derived-equivalence-for-abelian-category} and \ref{sec:resolving-der-eq} hold for certain Serre subcategories and thick subcategories of $Coh(X)$, by showing that they have the SR property.

Similar to the \Cref{defn:equivalent-filtrations-ideals}, we can define equivalent filtrations of ideal sheaves on a scheme $X$. 
\begin{defn}
    Let $\{ \I_t \}_{t \in \tiny \N}$ and $\{ \mathcal{J}_t \}_{t \in \tiny \N}$ be filtrations of ideal sheaves of $X$. We say the filtrations $\{ \I_t \}_{t \in \tiny \N}$ and $\{ \mathcal{J}_t \}_{t \in \tiny \N}$ are \emph{equivalent} if for each $t \in \N$ there exists $s$ such that $\mathcal{J}_s \subseteq \I_t$ and for each $s \in \N$, there exists $t$ such that $\I_t \subseteq \mathcal{J}_s$.
\end{defn}
\par For a closed set $V$ of $X$, let $\I_1$ and $\I_2$ be ideal sheaves of $\O_X$ defining closed subscheme structures on $V$. Note that for any affine open set $U$ of $X$, the radical ideals of $\I_1(U)$ and $\I_2(U)$ in $\O_X(U)$ are equal. Hence for any $t \in \N$, there exists $s \in \N$ such that $\I_2^s \subseteq \I_1^t$. In other words, the decreasing filtration of ideal sheaves $\{\I_1^t\}_{t \in \tiny \N}$ and $\{\I_2^t\}_{t \in \tiny \N}$ are equivalent.

\begin{defn}\label{defn:eventually-finite-A-dim}
    A closed set $V \subseteq X$ is said to have \emph{eventually finite $\CA$-dimension} if there exist ideal sheaves $\{\mathcal I_{t}\}_{t \in \N}$ of $\OX$ defining closed subscheme structures on $V$ such that \begin{enumerate}[(i)]
        \item  for any $s \in \N$, there exists $t$ for which $\mathcal I_{t} \subseteq \mathcal I^s_{1}$. \item\label{defn:ephd-condition2} $\dim_{\CA}(\mathcal I_{t}) < \infty$ for each $t \in \N$, that is, there is a finite resolution of $\mathcal I_{t}$ by objects in $\CA$.
    \end{enumerate}
\end{defn}
 In the above definition, since each $\I_t$ defines the closed subscheme structures on $V$, the condition $(i)$ is equivalent to the filtrations $\{\I_t\}_{t \in   \N}$ and $\{\I_1^s\}_{s \in   \N}$ being equivalent.
 \begin{rmk}
      Consider a projective scheme $Y = \proj(S)$ over a noetherian affine scheme, and $X \subseteq Y$ be open. 
      Any coherent sheaf $\F$ on $X$ can be extended to a coherent sheaf $\F'$ on $Y$ whose restriction to $X$ is $\F$ using \Cref{lem:Hartshorne-extension-of-sheaves}. By \cite[Cor 5.18]{Hartshorne}, there exists a surjection $\Oplus_{\text{finite}}\O_Y(n_i) \to \F'$ for some integers $n_i$. Since the restriction of a locally free sheaf on $Y$ to $X$ is again locally free on $X$, it follows that $Coh(X)$ has enough locally free sheaves on $X$. Therefore, whenever $X$ is a quasi-projective scheme over a noetherian affine scheme, $\V(X)$ is a resolving subcategory of $Coh(X)$.
  \end{rmk}

For $\F \in Coh(X)$ and $\CA = \V(X)$, we call $\dim_{  \CA}(\F)$ as the \emph{homological dimension} of $\F$, denoted by $\hd(\F)$. Note that for any $\F \in Coh(X)$, the homological dimension of $\F$ is at most $n$ iff $\sheafext^i(\F,\G) = 0$ for all $i >n$ and all sheaves of $\OX$-modules $\G$ (refer \cite[Ch 3, Ex. 6.5]{Hartshorne}). Therefore, the homological dimension can be computed locally as follows: $\hd(\F) = \sup_{x \in X}{\pd_{ \O_{X,x}}(\F_x)}$.
\begin{rmk}
    Note that any module over an isolated singularity has finite homological dimension over the punctured spectrum, but may not have finite homological dimension globally, e.g., the ideal $(X)$ in $\mathbb{C}[[X,Y]]/(XY)$. This shows that an ideal sheaf $\I$ in $X$ having finite $\V(X)$-dimension is a weaker property than the ideal sheaf $\I$ extended to the projective scheme $Y$ having finite $\V(Y)$-dimension. 
\end{rmk}
The following classes of closed sets have eventually finite homological dimension.
\begin{eg}\label{eg:efhd-1}
    \begin{enumerate}[(a)]
        \item If every point in a closed set $V$ is regular (i.e., $\O_{X,x}$ is a regular local ring for each $x \in V$), then $V$ has eventually finite $\V(X)$-dimension.
        \item\label{set-theoretic-CI} Let $S = \Oplus_{i \geq 0}$ be a graded ring which is finitely generated by $S_1$ as an $S_0$ algebra, $Y = \proj(S)$ and $X \subseteq Y$ open. Suppose $V = V(f_1,\dots,f_d) \cap X$ for homogeneous elements $f_1,\dots,f_d \in S$ and the restriction of $\stark_{\bullet}(f_1,\dots,f_d;R)$ to $X$ has no non-zero homologies in degree bigger than zero. Since taking powers of a regular sequence is again a regular sequence, it follows that $V$ has finite $\V(X)$-dimension.
    \end{enumerate}
\end{eg}
The following lemma connects the Cohen-Macaulay property for a local ring with the maximal ideal having eventually finite homological dimension and will be used later.
\begin{lemma}\cite[Corollary 5.4]{GanapathySane}\label{lem:CM-char-efpd}
    Let $(A,\m)$ be a noetherian local ring. Then $A$ is Cohen-Macaulay if and only if there exists a filtration of ideals $\{ I_t\}_{t \in   \N}$ such that the filtrations $\{ I_t\}$ and $\{\m^s\}$ are equivalent, and for each $t \in \N$, $I_t$ has finite projective dimension over $A$.
\end{lemma}
The scheme-theoretic version of the above lemma is the following.
\begin{lemma}
    Let $X$ be any noetherian scheme. Then, $X$ is a Cohen-Macaulay scheme if and only if every closed point has eventually finite $\V(X)$-dimension.
\end{lemma}
\begin{proof}
    Let $V = \{ x_0\}$ for a closed point $x_0$. Now using \Cref{lem:CM-char-efpd}, $\O_{X,x_0}$ is Cohen-Macaulay if and only if there exists a filtration of $\m_{x_0}$-primary ideals $\{I_t\}_{t \in   \N}$ such that for any $s \in \N$, there exists $t$ such that $I_t \subseteq \m_{x_0}^s$ and each $I_t$ has finite projective dimension over $\O_{X,x_0}$. Recall that for any ideal sheaf $\I$ defining a closed subscheme structure on $V$, ${\I}_{x_0}$ is an $\m_{x_0}$-primary ideal in the local ring $(\O_{X,x_0},\m_{x_0})$ and ${\I}_x$ is zero for all $x \neq x_0$. Thus, the ideal sheaves $\I_t$ defining the closed subscheme structures on $V$ whose stalk at $x_0$ is $I_t$ have finite homological dimension and also satisfy the condition \textit{(i)} in \Cref{defn:eventually-finite-A-dim}. Therefore, for a closed point $x_0 \in X$, $\O_{X,x_0}$ is a Cohen-Macaulay ring if and only if $\{x_0\}$ has eventually finite $\V(X)$-dimension.
\end{proof}
Recall that a scheme $X$ is said to have \emph{characteristic} $p$ for a prime number $p>0$, if $p$ is zero in $\OX$.
\begin{defn}
    Let $X$ be a noetherian scheme of prime characteristic $p>0$ and $\I$ an ideal sheaf in $X$. For $e \in \N$, the $e$-th \emph{Frobenius power} of $\I$, denoted by $\I^{[p^e]}$, is defined as the sheafification of the presheaf $U \mapsto I(U)^{[p^e]}$ for any affine open set $U$ in $X$.
\end{defn}
Note that $(\I^{[p^e]})_x = {(\I_x)}^{[p^e]}$ for all $x \in X$.
\begin{eg}\label{eg:efhd-char-p}
  Let $X$ be a noetherian scheme of prime characteristic $p>0$. Suppose $\I$ is an ideal sheaf of finite homological dimension. Then $\pd_{  \O_{X,x}}(\I_x) < \infty$ for all $x \in X$. By \cite[Theorem 1.7]{Peskine_Szpiro_thesis}, $\pd_{\tiny  \O_{X,x}}((\I_x)^{[p^e]}) = \pd_{\tiny \O_{X,x}}(\I_x) < \infty$ for all $e \in \N$. Since $\pd_{\tiny  \O_{X,x}}((\I^{[p^e]})_x) = \pd_{  \O_{X,x}}((\I_x)^{[p^e]})$ for all $x \in X$, $\hd_X(\I^{[p^e]}) = \hd_X(\I) < \infty$ for all $e \in \N$. Since the filtrations of ideal sheaves $\{ \I^{[p^e]} \}_{e \in \tiny \N}$ and $\{ \I^{n} \}_{n \in \tiny \N}$ are equivalent to each other, we can conclude that the closed set of $X$ corresponding to $\I$ has eventually finite $\V(X)$-dimension, whenever $\I$ has finite $\V(X)$-dimension.
\end{eg}
From now on, we assume $X$ is a quasi-projective scheme over an affine scheme, $\CA$ is a resolving subcategory of $Coh(X)$  which contains locally free sheaves on $X$ and also closed under shifts, and $\T$ is a thick subcategory of $Coh(X)$ containing $\CA$. Note that assuming $\V(X)$ to be contained in $\CA$ implies that $Coh(X)$ has enough $\CA$-objects. Now we state one of the main theorems of the section. 
    \begin{thm}\label{thm:SR-prop-holds-cohx}
        Let $V$ be a closed set, $\L = \L_{V}$ and $\T$ a thick subcategory of $Coh(X)$. Suppose $\T$ contains a resolving subcategory $\CA$ of $Coh(X)$ which contains $\OX$ and is closed under shifts, and $V$ has eventually finite $\CA$-dimension. Then $\L$ has the SR property with respect to $\T$.
    \end{thm}
    The theorem is a consequence of \Cref{lem:prep-for-SR-prop} below. Before stating the lemma, we set up some notations.

Define $g(V,\CA)$ as follows:
For a closed set $V$ and an ideal sheaf $\I$ defining a closed subscheme structure on $V$, let $S(V) = \{ \{\I_t\}_{t \in   \N} \mid \{\I_t\}_{t \in   \N} \text{\ is\ equivalent\ to\ } \{\I^n\}_{n \in   \N}\}$. Now we define $$g(V,\CA) = \inf\limits_{\{\tiny \I_t\} \in S(V)} \limsup\limits_{t \in \tiny \N} \{ \dim_{\tiny \CA}(\OX/\I_t)  \}.$$
\begin{lemma}\label{lem:prep-for-SR-prop}
    Let $V \subseteq X$ be a closed set, $\L = \L_{V}$, $\G_\bullet \in Ch^b_{\tiny \L}(Coh(X))$ with $\min(\G_\bullet) = m$ and $\lambda \colon \G_m \to \mathcal Q$ a morphism with $\mathcal Q \in \L$. Suppose $V$ has eventually finite $\CA$-dimension. Then there exists a chain complex $\CT_\bullet \in Ch^{b}_{\tiny \L}(\CA)$ and a chain complex map $\alpha\colon \CT_\bullet \to \G_\bullet$ such that $(\CT_\bullet,\alpha)$ is a strong reducer of $(\G_\bullet, \lambda)$. Further if $g(V,\CA) < \infty$, then $\dim_{\tiny \CA}(\H_m(\CT_\bullet)) = g(V,\CA)$.
    \end{lemma}
\begin{proof}
From \Cref{lem:foxby-morphism-for-schemes}, we get $\CU_\bullet \in Ch^{+}_{\tiny \L}(\V(X))$ and a chain complex map $\beta \colon \CU_\bullet \to \G_\bullet$ satisfying the properties $(1)- (5)$ as in \Cref{lem:foxby-morphism-for-schemes}. In particular,
$\H_m(\CU_\bullet) \cong (\OX/\I) \otimes \OX(n)$, for some ideal sheaf $\I$ defining a closed subscheme structure on $V$ and $n \in \N$. 
Since $V$ has eventually finite $\CA$-dimension, there exists a filtration $\{ \I_t \}$ equivalent to $\{\I^s\}$ where $\dim_{\tiny \CA}(\OX/\I_t) < \infty$ for all $t \in \N$ and $g(V,\CA) = \limsup\limits_{t \in \tiny \N} \{ \dim_{\tiny \CA}(\OX/\I_t)\}$. Note that there exists $t_0 \in \N$ such that for all $t \geq t_0$, $\I_t \subseteq \I$.
    Let $s \geq t_0$. Since $\dim_{\tiny \CA}((\OX/\I_s) \otimes \OX(n)) = \dim_{\tiny \CA}(\OX/\I_s) = r$ (say), using \Cref{rmk:resolving-subcategories}(2), there exists an $\CA$-resolution of $(\OX/\I_s) \otimes  \OX(n)$ that has a morphism to $\CU_\bullet$ which lifts the surjection $(\O_X/\mathcal I_s)\otimes \OX(n) \to (\O_X/\mathcal I)\otimes \OX(n)$. Call $\CT_\bullet$ to be the soft truncation of the chosen resolution of $(\O_X/\mathcal I_s) \otimes \OX(n)$ at $r+m$, and $\alpha \colon \CT_\bullet \to \G_\bullet$ be the induced chain complex map. Now it is easy to see that $(\CT_\bullet,\alpha)$ is a strong reducer of $(\G_\bullet, \lambda)$. 
    \par If $g(V,\CA) < \infty$, then choose $s$ such that $s \geq t_0$ and $g(V,\CA) = \dim_{\tiny \CA}(\OX/\I_s) = r$. Since $r = \dim_{\tiny \CA}((\OX/\I_s) \otimes \OX(n))$, we get $\dim_{\tiny \CA}(\H_m(\CT_\bullet)) = g(V,\CA)$.
    \end{proof}
    The following corollary is a strengthening of \Cref{thm:SR-prop-holds-cohx} and its proof is exactly on the same lines as  \Cref{rmk:SRproperty-extension-graded}.
    \begin{cor}\label{rmk:SRproperty-extension-scheme}
    Let $\T$ and $\CA$ be as in \Cref{thm:SR-prop-holds-cohx}. Suppose $W$ is a specialization closed set of $X$ such that for each closed set $W'\subseteq W$, there exists a closed set $W''$ of $X$ with $W' \subseteq W'' \subseteq W$ and $W''$ has eventually finite $\CA$-dimension. Then the Serre subcategory $\L_W$ consisting of objects $ \{ \F \in Coh(X) \mid \Supp_X(\F) \subseteq W\}$ has the SR property with respect to $\T$. 
\end{cor}
    As a consequence of \Cref{lem:prep-for-SR-prop}, we get the following corollary.
    \begin{cor}\label{cor:surjectivity-from-perfect-sheaf}
        Let $V \subseteq X$ be a closed set having eventually finite $\CA$-dimension and $\L = \L_{V}$. Then for each $\F \in \L$, there exist $\F' \in \L$ with an epimorphism $\F' \to \F$ and $\dim_{\tiny \CA}(\F') \leq g(V,\CA)$.
    \end{cor}
  \begin{proof}
     If $g(V,\CA)$ is infinite, there is nothing to prove. So we assume $g(V,\CA)< \infty$. Applying \Cref{lem:prep-for-SR-prop} to $\G_\bullet = 0 \to \F \to 0$ and $\lambda = \operatorname{id}_{\F}$, there exists $\F' \coloneqq \H_0(\CT_\bullet)$ with an epimorphism $\F' \to \F$ with $\dim_{\tiny \CA}(\F') = g(V,\CA)$.
  \end{proof}
  Applying \Cref{thm:SR-prop-holds-cohx} to \Cref{thm:derived-equivalence-AbarcapL-AbarsuppL-AsuppL}, and the above corollary to \Cref{thm:derived-eq-keller-abelian-category}, yields the following: 
  \begin{thm}\label{thm:der-eq-scheme}
      Let $X$ be a quasi-projective scheme over a noetherian affine scheme, $\CA$ a resolving subcategory of $Coh(X)$ which contains $\OX$ and is closed under shifts, $V \subseteq X$ a closed set, $\L = \L_V$, and $g = g(V,\CA)$. Suppose $V$ has eventually finite $\CA$-dimension. Then $$D^b(\overline{\CA}^{\leq g} \cap \L) \simeq D^b(\overline{\CA} \cap \L) \simeq D^b_{\tiny \L}(\overline{\CA}) \simeq D^b_{\tiny \L}(\CA).$$
  \end{thm}


 Recall that for a sheaf $\F$, the grade of $\F$ is defined as $$\grade_X(\F) \coloneqq \min \{ i \mid \sheafext^i_X(\F,\OX) \neq 0 \}.$$ For a closed set $V \subseteq X$, define $\grade_X(V) \coloneqq \grade(\OX/\I_V)$ for an ideal sheaf $\I_V$ defining a closed subscheme on $V$. It is well defined and does not depend on the choice of $\I_V$. For more details on the definition and its well-definedness, we refer to \cite[Appendix A]{Mandal2025}.  
\begin{defn}
    We say a sheaf $\F \in Coh(X)$ is \emph{perfect} if $\grade_X(\F) = \hd_X(\F)$. An ideal sheaf $\I$ is said to be \emph{perfect} if the sheaf $\OX/\I$ is perfect.
\end{defn}

Similar proofs as those for graded modules yield that $\grade_X(V) \leq g(V,\V(X))$, and when $\I_V$ is perfect and the characteristic of $X$ is prime, $g(V,\V(X)) = \grade_X(V)$. Thus, we have the following theorem.
\begin{thm}\label{thm:der-eq-schemes-locally-free}
Let $X$ be a quasi-projective scheme over a noetherian affine scheme, $V \subseteq X$ a closed set, and $\L = \L_{V}$. Suppose $X$ has prime characteristic and $\I_V$ is perfect with $\grade_X(\I_V) = g$ for some ideal sheaf $\I_V$ defining a closed subscheme structure on $V$. Then $$D^b(\overline{\V(X)}^{\leq g} \cap \L) \simeq D^b( \overline{\V(X)} \cap \L) \simeq D^b_{\tiny \L}( \overline{ \V(X)}) \simeq D^b_{\tiny \L}(\V(X)).$$
\end{thm}
\begin{proof}
    Since $\I_V$ is perfect and $X$ has characteristic $p$, by \Cref{eg:efhd-char-p}, $V$ has eventually finite $\V(X)$-dimension. Using $g(V,\V(X)) = \grade_X(V)$ and applying \Cref{thm:der-eq-scheme}, we get the derived equivalence as required. 
\end{proof}

\section{Results about $\Kth$ and $\GWth$ spectra}\label{sec:result-on-K-theory-and-G-theory}
\par In this section, we briefly summarize the consequences of the derived equivalences in the previous sections for (non-connective) $\Kth$-theory and $\GWth$-theory. We recall the notations. $X$ is a quasi-projective scheme over a noetherian affine scheme, $V$ is a closed set, $U = X\backslash V$ and $\L = \L_{V}$ is the corresponding Serre subcategory of $Coh(X)$. Then there is a sequence of triangulated categories which is exact upto factors
$$D^b_{\tiny \L}(\V(X)) \to D^b(\V(X)) \to D^b(\V(U)).$$ For a proof, we refer to \cite[Proposition A.4.7]{Marco-algebraic-k-theory}. After applying the $\Kth$ and $\GWth$ functors to the above sequence, we get a homotopy fibration sequence. The previously proved derived equivalences then yield a description of the homotopy fiber in that sequence as described ahead.
\par More specifically, we state below the result obtained by applying the non-connective K-theory spectrum functor $\Kth$.
\begin{thm}\label{thm:K-thry-fibration}
    Let $X$ be a quasi-projective scheme over a noetherian affine scheme, $V$ a closed set having eventually finite $\V(X)$-dimension, $\L = \L_V$ and $g = g(V,\V(X))$. Then $$\Kth(\overline{\V(X)}^{\leq g} \cap \L) \to \Kth(\V(X)) \to \Kth(\V(U)) $$ is a homotopy fibration in the category $\spc$ of spectra.
\end{thm}
A similar result holds for the $\GWth$ bispectra. We first introduce the notations required for the statement. Let $X$ be a quasi-projective scheme over a noetherian ring $A$ with $1/2 \in A$, $\mathcal L$ an invertible sheaf on $X$, and $\I_V$ an ideal sheaf defining a closed subscheme structure on $V$. Suppose $\I_V$ is perfect with $\grade_X(\I_V) = g$. For $\F \in \overline{\V(X)}^{\leq g} \cap \L$, consider the sheaf $\sheafext^g_X(\F,\mathcal L)$. Since $\I_V$ is perfect of grade $g$, $\F$ is also perfect of grade $g$. Then it follows that $\sheafext^g_X(\F,\mathcal L) \in \overline{\V(X)}^{\leq g} \cap \L$. Hence we have a duality on $\overline{\V(X)}^{\leq g} \cap \L$ given by $\F$ mapping to $\sheafext^{g}_X(\F,\mathcal L)$. 
    \par Also there is a duality on $D^b_{\L}(\V(X))$ given by $\Sigma^{-g}\sheafhom_X(\_,\mathcal L)$. Further, when $X$ has prime characteristic, the equivalence $D^b(\overline{\V(X)}^{\leq g} \cap \L)  \simeq D^b_{\L}(\V(X))$ in \Cref{thm:der-eq-schemes-locally-free} given by taking resolutions also preserves dualities. As a consequence, we get the following:

\begin{thm}\label{thm:G-thry-fibration}
    Let $X$ be a quasi-projective scheme over $Spec(A)$ for a noetherian ring $A$ with $1/2 \in A$, $\mathcal L$ an invertible sheaf on $X$, $V \subseteq X$ a closed set, $U = X \setminus V$ an open set, and $\L = \L_V$. Suppose $X$ has prime characteristic and $\I_V$ is perfect with $\grade_X(\I_V) = g$ for some ideal sheaf $\I_V$ defining a closed subscheme structure on $V$. Then the sequence
\[
\GWth^{[-g]}(\overline{\V(X)}^{\leq g} \cap \L) \to \GWth(\V(X)) \to \GWth(\V(U))
\]
is a homotopy fibration of in the category $\bisp$ of bispectra.
\end{thm}
The agreement theorem, excision theorem and the Mayer-Vietoris sequences in \cite[Theorems 6.10 to 6.17]{Mandal2025} for $\Kth$-theory spectra and $\GWth$ bispectra also hold in our case, where $\Kth$-theory statements need the hypotheses in \Cref{thm:K-thry-fibration} and $\GWth$-theory statements need the hypotheses in \Cref{thm:G-thry-fibration}.

Regarding the proofs, note that whenever $V$ has eventually finite $\V(X)$-dimension, applying \Cref{thm:der-eq-scheme} with $\CA = \V(X)$ yields $D^b(\overline{\V(X)}^{\leq g} \cap \L) \simeq D^b_{\tiny \L}(\V(X))$ where $g = g(V,\V(X))$. This equivalence implies all the results in this section and we skip the proofs since they follow in the same way as in \cite[Sections 4 to 6]{Mandal2025}. For the convenience of the reader, we mention the differences between our notation and the one in \cite{Mandal2025}: In \cite{Mandal2025}, $Z$ is a closed set and $Coh^{Z}(X)$ is the Serre subcategory of all sheaves supported on $Z$, whereas we use $V$ as a closed set and $\L_V$ is the corresponding Serre subcategory. Also, $\mathbb M(X), \mathbb M^Z(X), C\mathbb M^Z(X)$ and $\mathscr D^b_Z(\V(X))$ in \cite{Mandal2025} appear in our article as $\overline{\V(X)}, \overline{\V(X)}\cap \L_V, \overline{\V(X)}^{\leq g} \cap \L_V$ and $D^b_{\tiny \L_V}(\V(X))$ respectively, where $g = \grade_X(V)$.

\section*{Acknowledgments}
The authors are grateful to Henning Krause for valuable discussions during his visit to IIT Madras. The authors thank Rahul Gupta for providing useful suggestions concerning the notation used in the article. Ganapathy Krishnamoorthy would like to acknowledge the support from the Prime Minister's Research Fellowship (PMRF) scheme for carrying out this research work.


\begin{thebibliography}{10}

\bibitem{Auslander-Bridger}
Maurice Auslander and Mark Bridger.
\newblock {\em Stable module theory}, volume No. 94 of {\em Memoirs of the American Mathematical Society}.
\newblock American Mathematical Society, Providence, RI, 1969.

\bibitem{Fossum-Foxby74}
Robert Fossum and Hans-Bj{\o}rn Foxby.
\newblock The category of graded modules.
\newblock {\em Math. Scand.}, 35:288--300, 1974.

\bibitem{Foxby-Halvorsen}
Hans-Bj{\o}rn Foxby and Esben~Bistrup Halvorsen.
\newblock Grothendieck groups for categories of complexes.
\newblock {\em J. K-Theory}, 3(1):165--203, 2009.

\bibitem{GanapathySane}
K.~Ganapathy and Sarang Sane.
\newblock Derived equivalences via {T}ate resolutions.
\newblock {\em J. Algebra}, 672:89--119, 2025.

\bibitem{Grothendieck}
A.~Grothendieck.
\newblock Groupes de classes des categories abeliennes et triangulees. complexes parfaits.
\newblock In Luc Illusie, editor, {\em S{\'e}minaire de G{\'e}om{\'e}trie Alg{\'e}brique du Bois-Marie 1965--66 SGA 5}, pages 351--371, Berlin, Heidelberg, 1977. Springer Berlin Heidelberg.

\bibitem{Hartshorne}
Robin Hartshorne.
\newblock {\em Algebraic geometry}, volume No. 52 of {\em Graduate Texts in Mathematics}.
\newblock Springer-Verlag, New York-Heidelberg, 1977.

\bibitem{Ruben-Sondre-Roosmalen}
Ruben Henrard, Sondre Kvamme, and Adam-Christiaan van Roosmalen.
\newblock Auslander's formula and correspondence for exact categories.
\newblock {\em Adv. Math.}, 401:Paper No. 108296, 65, 2022.

\bibitem{Keller}
Bernhard Keller.
\newblock On the cyclic homology of exact categories.
\newblock {\em J. Pure Appl. Algebra}, 136(1):1--56, 1999.

\bibitem{HenningKrause}
Henning Krause.
\newblock {\em Homological theory of representations}, volume 195 of {\em Cambridge Studies in Advanced Mathematics}.
\newblock Cambridge University Press, Cambridge, 2022.

\bibitem{Mandal2015}
Satya Mandal.
\newblock Foxby-morphism and derived equivalences.
\newblock {\em J. Algebra}, 440:113--127, 2015.

\bibitem{Mandal2025}
Satya Mandal.
\newblock Localization problems of {Q}uillen.
\newblock {\em J. Algebra}, 680:205--266, 2025.

\bibitem{Peskine_Szpiro_thesis}
C.~Peskine and L.~Szpiro.
\newblock Dimension projective finie et cohomologie locale. {A}pplications \`a la d\'{e}monstration de conjectures de {M}. {A}uslander, {H}. {B}ass et {A}. {G}rothendieck.
\newblock {\em Inst. Hautes \'{E}tudes Sci. Publ. Math.}, (42):47--119, 1973.

\bibitem{SandersSane}
William~T. Sanders and Sarang Sane.
\newblock Finite homological dimension and a derived equivalence.
\newblock {\em Trans. Amer. Math. Soc.}, 369(6):3911--3935, 2017.

\bibitem{Marco-algebraic-k-theory}
Marco Schlichting.
\newblock Higher algebraic {$K$}-theory.
\newblock In {\em Topics in algebraic and topological {$K$}-theory}, volume 2008 of {\em Lecture Notes in Math.}, pages 167--241. Springer, Berlin, 2011.

\bibitem{weibelhomological}
Charles~A. Weibel.
\newblock {\em An introduction to homological algebra}, volume~38 of {\em Cambridge Studies in Advanced Mathematics}.
\newblock Cambridge University Press, Cambridge, 1994.

\end{thebibliography}
\end{document}